\documentclass[11pt,reqno]{amsart}
\usepackage[all]{xy}
\usepackage{amsfonts}
\usepackage{amssymb}
\usepackage{graphicx}
\graphicspath{{FIGURES/}}
\usepackage[latin1]{inputenc}

\usepackage{amsmath,amssymb,amsthm, latexsym, amscd}
\usepackage{tikz-cd}

\usepackage{hyperref}

\usepackage{mathtools}
\DeclarePairedDelimiter{\ceil}{\lceil}{\rceil}

\usepackage{mathrsfs}
\usepackage{bbm}

\usepackage{color}

\numberwithin{equation}{section}

\newtheorem{theorem}{Theorem}[section]
\newtheorem{proposition}[theorem]{Proposition}
\newtheorem{corollary}[theorem]{Corollary}
\newtheorem{lemma}[theorem]{Lemma}

\theoremstyle{definition}
\newtheorem{definition}[theorem]{Definition}
\newtheorem{example}[theorem]{Example}
\newtheorem{remark}[theorem]{Remark}

\newtheorem{conjecture}[theorem]{Conjecture}

\def\<{\langle}
\def\>{\rangle}

\def\a{\alpha}
\def\b{\beta}
\def\g{\gamma}

\def\e{\varepsilon}
\def\l{{\lambda}}
\def\r{\rho}

\def\t{\tau}
\def\te{\theta}

\def\AA{{\mathbb A}}

\def\N{{\mathbb N}}
\def\R{{\mathbb R}}
\def\Z{{\mathbb Z}}
\def\Q{{\mathbb Q}}

\def\ent{{\rm ent}}

\def\diam{\mathop{\rm diam}\nolimits}
\def\dist{\mathop{\rm dist}\nolimits}

\def\dist{\mathop{\rm dist}\nolimits}

\def\1{\mathbf 1}
\def\sys{\operatorname{sys}}

\newcommand{\ie}{{\it i.e.}}
\newcommand{\eg}{{\it e.g.}}

\newcommand{\vol}{{\rm vol}}
\newcommand{\length}{{\rm length}}

\newcommand{\Cone}{{\rm Cone}}

\def\N{{\mathbb N}}

\long\def\forget#1\forgotten{} %

\begin{document}

\title{Volume entropy semi-norm}

\author[I.~Babenko]{Ivan Babenko}
\author[S.~Sabourau]{St\'ephane Sabourau}

\thanks{Partially supported by the ANR project Min-Max.}

\address{Universit\'e Montpellier II, CNRS UMR 5149,
Institut Montpelli\'erain Alexander Grothendieck,
Place Eug\`ene Bataillon, B\^at. 9, CC051, 34095 Montpellier CEDEX
5, France} 

\email{babenko@umontpellier.fr}

\address{\parbox{\linewidth}{LAMA, Univ Paris Est Creteil, Univ Gustave Eiffel, UPEM, CNRS, F-94010, Cr\'eteil, France \\
CRM (UMI3457), CNRS, Univ Montr\'eal, Case postale 6128, Montr\'eal, QC, Canada}}



\email{stephane.sabourau@u-pec.fr}



\begin{abstract}
We introduce the volume entropy semi-norm in real homology and show that it satisfies functorial properties similar to the ones of the simplicial volume.
Answering a question of M.~Gromov, we prove that the volume entropy semi-norm is equivalent to the simplicial volume semi-norm in every dimension.
We also establish a roughly optimal upper bound on the systolic volume of the multiples of any homology class.
\end{abstract}

\maketitle

\tableofcontents

\section{Introduction}

This article deals with topology-geometry interactions and the comparison of functorial geometric semi-norms on the real homology groups of topological spaces.
In his book~\cite[\S G$_+$-H$_+$]{Gromov99}, M.~Gromov pointed out directions where such geometric semi-norms might arise in relation with curvature (\eg, sectional, Ricci, scalar) properties for instance.
Yet, the corresponding invariants have not been properly defined or studied as semi-norms, except for the simplicial volume which is a purely topological invariant.
As part of this program to investigate the interactions between geometry and topology through the study of functorial geometric semi-norms, we introduce the volume entropy semi-norm in real homology and carry out a systematic study of this invariant.
The definition of this semi-norm relies on the volume entropy, a geometric invariant of considerable interest closely related to the dynamics of the geodesic flow and the growth of the fundamental groups.
We show in particular that the volume entropy semi-norm shares similar functorial properties with the simplicial volume (also called the Gromov semi-norm) and prove that the two semi-norms are actually equivalent, which answers a question of M.~Gromov.
A connection with systolic geometry is also revealed along the way.


\medskip


Let $M$ be a connected closed $m$-dimensional manifold with a Riemannian metric~$g$.
Let $H \lhd \pi_1(M)$ be a normal subgroup of the fundamental group of~$M$.
The \emph{volume entropy} (or simply \emph{entropy}) of~$(M,g)$ relative to~$H$, denoted by~$\ent_H(M,g)$, is the exponential growth rate of the volume of balls in the Riemannian covering~$M_H$ corresponding to the normal subgroup $H \lhd \pi_1(M)$, that is, \mbox{$\pi_1(M_H) = H$.}
More precisely, it is defined as
\begin{equation}\label{eq:entropy.comp.space}
\ent_H(M, g) =\lim_{R \to \infty} \frac{1}{R} \log[ \vol \, B_H(R)]
\end{equation}
where $B_H(R)$ is a ball of radius~$R$ centered at any point in the covering~$M_H$.
The limit exists and does not depend on the center of the ball. 
When $H$ is trivial, the covering~$M_H$ is the universal covering~$\widetilde{M}$ of~$M$ and we simply denote its volume entropy by~$\ent(M,g)$ without any reference to~$H$.
Note that 
\[
\ent_H(M,g) \leq \ent(M,g)
\]
for every normal subgroup~$H \lhd \pi_1(M)$.
The definition extends to connected closed pseudomanifolds, see Definition~\ref{def:pseudomanifold}, to connected finite graphs and more generally, to finite simplicial complexes with a length metric.

\medskip

The importance of this notion was first noticed by Efremovich~\cite{Efrem53}.
Subsequently, \v{S}varc~\cite{Shvarts55} and Milnor~\cite{Milnor68} related the growth of the volume of balls in the universal covering~$\widetilde{M}$ to the growth of the fundamental group~$\pi_1(M)$ of~$M$.
Note that the volume entropy of a connected closed Riemannian manifold is positive if and only if its fundamental group has exponential growth. 
The connexion with the dynamics of the geodesic flow was established by Dinaburg~\cite{Dinaburg71} and Manning~\cite{Manning79}.
More specifically, the volume entropy bounds from below the topological entropy of the geodesic flow on a connected closed Riemannian manifold and the two invariants coincide when the manifold is nonpositively curved, see~\cite{Manning79}.

\medskip

The \emph{minimal volume entropy} of a closed $m$-pseudomanifold~$M$ relative to a normal subgroup~$H \lhd \pi_1(M)$ is defined as
\[
\omega_H(M) = \inf_g \, \ent_H(M,g) \, \vol(M,g)^\frac{1}{m}
\]
where $g$ runs over the space of all piecewise Riemannian metrics on~$M$.
For convenience, we also introduce 
\[
\Omega_H(M,g) = \ent_H(M,g)^m \, \vol(M,g) 
\]
and
\[
\Omega_H(M) = \inf_g \, \ent_H(M,g)^m \, \vol(M,g)
\]
where $g$ runs over the space of all piecewise Riemannian metrics on~$M$.
As previously, if $H$ is trivial, we drop the subscript~$H$.

\medskip

As an example, the minimal volume entropy of a closed $m$-manifold~$M$ which carries a hyperbolic metric is attained by the hyperbolic metric and is equal to $(m-1) \, \vol(M,{\rm hyp})^\frac{1}{m}$, see~\cite{katok},\cite{BCG91} for $m=2$ and~\cite{BCG95} for~$m \geq 3$.
Furthermore, the minimal volume entropy of a closed manifold which carries a negatively curved metric is positive, see~\cite{gro82}.

\medskip

For a connected closed orientable $m$-manifold~$M$, the minimal volume entropy of~$M$ is a homotopy invariant, see~\cite{B93}, which only depends on the image $h_*([M]) \in H_m(\pi_1(M);\Z)$ of the fundamental class of~$M$ by the homomorphism induced by the classifying map $h:M \to K(\pi_1(M),1)$ of~$M$ in homology, see~\cite{Brunnbauer08}.
We say that the minimal volume entropy is a \emph{homological invariant}.

\medskip

This homological invariance leads us to consider the volume entropy of a homology class as follows.
Given a path-connected topological space~$X$,
the \emph{volume entropy of a homology class} \mbox{${\bf a} \in H_m(X;\Z)$} is defined as 
\begin{equation} \label{eq:omegaa}
\omega({\bf a}) = \inf_{(M,f)} \omega_{\ker f_*}(M)
\end{equation}
where the infimum is taken over all $m$-dimensional \emph{geometric cycles} $(M,f)$ representing~${\bf a}$, that is, over all maps $f:M \to X$ from an orientable connected closed $m$-pseudomanifold~$M$ to~$X$ such that $f_*([M]) = {\bf a}$.
As previously, we define
\[
\Omega({\bf a}) = \omega({\bf a})^m.
\]

For every map $f: X \to Y$ between two path-connected topological spaces and every \mbox{${\bf a} \in H_m(X,\Z)$}, we have
\begin{equation} \label{eq:Omegaf}
\Omega(f_*({\bf a})) \leq \Omega({\bf a}).
\end{equation}
By~\cite[Theorem~10.2]{Brunnbauer08}, every orientable connected closed $m$-manifold~$M$ with~$m \geq 3$ satisfies
\begin{equation} \label{eq:thm10.2}
\Omega(M) = \Omega(h_*([M])
\end{equation}
where $h:M \to K(\pi_1(M),1)$ is the classifying map of~$M$.

\medskip

The following result shows that $\Omega$ induces a pseudo-distance in homology.

\begin{theorem} \label{theo:semi}
Let $X$ be a path-connected topological space.
Then for every ${\bf a}, {\bf b} \in H_m(X;\Z)$, we have
\[
\Omega({\bf a} + {\bf b}) \leq \Omega({\bf a}) + \Omega({\bf b}).
\]
In particular, the quantity $\Omega({\bf a} - {\bf b})$ defines a pseudo-distance between ${\bf a}$ and ${\bf b}$ in~$H_m(X;\Z)$,
\end{theorem}

Thus, for every homology class ${\bf a} \in H_m(X;\Z)$, the sequence $\Omega(k \, {\bf a})$ is sub-additive.
As a result, we can apply the following stabilization process and define
\begin{equation}\label{eq:semi-norme-entropie}
\Vert {\bf a} \Vert_E = \mathop{\lim}\limits_{k \rightarrow \infty}\frac{\Omega(k \, {\bf a})}{k}.
\end{equation}
Note that $\Vert \cdot \Vert_E$ is homogenous, that is, $\Vert k \, {\bf a} \Vert_E = \vert k \vert \, \Vert {\bf a} \Vert_E$ for every $k \in \Z$.
By homogeneity and density of~$H_m(X;\Q)$ in~$H_m(X;\R)$, this functional extends to a functional on~$H_m(X;\R)$, still denoted by~$\Vert \cdot \Vert_E$.

\medskip

For an orientable connected closed $m$-manifold~$M$, define
\[
\Vert M \Vert_E = \Vert [M] \Vert_E
\]
where $[M] \in H_m(M;\Z)$ is the fundamental class of~$M$.

\medskip

The following result, which is a direct consequence of Theorem~\ref{theo:semi}, justifies the use of the term \emph{volume entropy semi-norm} to designate the functional~$\Vert \cdot \Vert_E$.

\begin{corollary} \label{coro:seminorm0}
Let $X$ be a path-connected topological space.
Then the functional~$\Vert \cdot \Vert_E$ is a semi-norm on~$H_m(X;\R)$.
\end{corollary}

Functorial properties of the volume entropy semi-norm are described in Section~\ref{sec:functorial}, where it is shown that it is a homological invariant. \\

The simplicial volume is a much-studied topological invariant sharing similar properties with the volume entropy semi-norm.
Let us recall its definition and its basic properties, referring to~\cite{gro82} for foundational constructions and results regarding this invariant.
Let $X$ be a path-connected topological space.
Every real singular $m$-chain~$c \in C_m(X;\R)$ of~$X$ is a real linear combination of singular simplices $f_s:\Delta^m \to X$, that is,
\[
c=\sum_s r_s \, f_s
\]
where $r_s \in \R$.
The \emph{$\ell_1$-norm} on the real chain complex is defined as
\[
\Vert c \Vert_1 = \sum_s |r_s|.
\]
The \emph{simplicial volume} of a real homology class ${\bf a} \in H_m(X;\R)$ is defined as
\[
\Vert {\bf a} \Vert_\Delta = \inf_c \, \Vert c \Vert_1
\]
where the infimum is taken over all real singular $m$-cycles~$c$ representing~${\bf a}$.
The simplicial volume of an integral homology class is defined as the simplicial volume of the corresponding real homology class.
It is clear that the simplicial volume~$\Vert \cdot \Vert_\Delta$ is a functorial semi-norm on~$H_m(X;\R)$.
As previously, for an orientable connected closed $m$-manifold~$M$, we let
\[
\Vert M \Vert_\Delta = \Vert [M] \Vert_\Delta
\]
where $[M] \in H_m(M;\Z)$ is the fundamental class of~$M$.
By~\cite[\S3.1]{gro82}, the simplicial volume is a homological invariant.

\medskip

The following inequality of M.~Gromov~\cite[p.~37]{gro82} connects the minimal volume entropy of an orientable connected closed manifold to its simplicial volume (see~\cite{sab17} for other topological conditions ensuring the positivity of the minimal volume entropy through a different approach).
Namely, every orientable connected closed $m$-manifold~$M$ satisfies
\begin{equation} \label{eq:gro}
\Omega(M) \geq c_m \, \Vert M \Vert_\Delta
\end{equation}
for some positive constant~$c_m$ depending only on~$m$.
Extending this inequality to the semi-norm level, see Theorem~\ref{theo:E>S}, we obtain that every homology class ${\bf a} \in H_m(X;\R)$ of a path-connected topological space~$X$ satisfies
\[
\Vert {\bf a} \Vert_E \geq c_m \, \Vert {\bf a}  \Vert_\Delta
\]
with the same constant~$c_m$ as in~\eqref{eq:gro}.

\medskip

To illustrate our ignorance about the basic topology-geometry interaction, M.~Gromov \cite[\S5.41, p.~310]{Gromov99} raised the question whether the two semi-norms $\Vert \cdot \Vert_E$ and $\Vert \cdot \Vert_\Delta$ are equivalent, that is, whether a reverse inequality to~\eqref{eq:gro} holds.

\medskip

The following result affirmatively answers this question.

\begin{theorem} \label{theo:main}
Let $X$ be a path-connected topological space.
For every positive integer~$m$, there exist two positive constants $c_m$ and~$C_m$ depending only on~$m$ such that every homology class ${\bf a} \in H_m(X;\R)$ satisfies
\[
c_m \, \Vert {\bf a}  \Vert_\Delta \leq \Vert {\bf a} \Vert_E \leq C_m \, \Vert {\bf a}  \Vert_\Delta.
\]
\end{theorem}

We immediately deduce the following corollary.

\begin{corollary} \label{coro:vanishing}
Let $X$ be a path-connected topological space and ${\bf a} \in H_m(X;\R)$ be a homology class.
Then $\Vert {\bf a} \Vert_E$ vanishes
if and only if $\Vert {\bf a} \Vert_\Delta $ vanishes.

In particular, for every orientable connected closed manifold~$M$, the volume entropy semi-norm~$\Vert M \Vert_E$ is zero if and only if the simplicial semi-norm~$\Vert M \Vert_\Delta$ is zero.
\end{corollary}

Combining the recent result~\cite{HL19} on the spectrum of the simplicial volume with Theorem~\ref{theo:main}, we immediately deduce that the volume entropy semi-norm~$\Vert \cdot \Vert_E$ is not bounded away from zero in dimension greater than~$3$. (In dimension $2$ and~$3$, there is a gap in the simplicial volume spectrum, and so in the volume entropy spectrum by Theorem~\ref{theo:main}.)
More generally, we have the following result.

\begin{corollary}
Let $m \geq 4$ be an integer.
Then the set of all volume entropy semi-norms~$\Vert M \Vert_E$ of orientable connected closed $m$-manifolds~$M$ is dense in~$[0,\infty)$.
\end{corollary}

\medskip

Along the proof of Theorem~\ref{theo:main}, we establish a new result in systolic geometry, answering another question of M.~Gromov.
Before stating this result, we need to introduce various notions.

\medskip

Let $M$ be a closed $m$-dimensional manifold or pseudomanifold with a (piecewise) Riemannian metric~$g$.
Let $f:M \to X$ be a map to a topological space~$X$.
The \emph{systole} of~$M$ relative to~$f$, denoted by~$\sys_f(M,g)$, is defined as the least length of a loop~$\gamma$ in~$M$ whose image by~$f$ is noncontractible in~$X$.
The \emph{systolic volume} of~$M$ relative to~$f$ is defined as
\begin{equation} \label{eq:sigmaf}
\sigma_f(M) = \inf_g \, \frac{\vol(M,g)}{\sys_f(M,g)^m}
\end{equation}
where the infimum is taken over all (piecewise) Riemannian metrics~$g$ on~$M$.
When $f:M \to X$ is $\pi_1$-injective, for instance, when $f:M \to K(\pi_1(M),1)$ is the classifying map of~$M$, we simply denote its systolic volume by~$\sigma(M)$ without any reference to~$f$.
By~\cite{B06}, \cite{B06b}, \cite{Brunnbauer08}, the systolic volume is a homological invariant.

\medskip

As with~\eqref{eq:omegaa}, the \emph{systolic volume of a homology class} ${\bf a} \in H_m(X;\Z)$, where $X$ is a path-connected topological space, is defined as
\begin{equation} \label{eq:sigmaa}
\sigma({\bf a}) = \inf_{(M,f)} \sigma_f(M)
\end{equation}
where the infimum is taken over all $m$-dimensional geometric cycles~$(M,f)$ representing~${\bf a}$.





\medskip

Every closed genus~$g$ surface~$\Sigma_g$ satisfies 
\[
A \, \frac{g}{(\log g)^2} \leq \sigma(\Sigma_g) \leq B \, \frac{g}{(\log g)^2}
\]
for some universal positive constants $A$ and~$B$.
The first inequality was established by M.~Gromov \cite{gro83}, \cite{gro96}.
The second inequality was proved by P.~Buser and P.~Sarnak in~\cite{BS94}, where they constructed hyperbolic genus~$g$ surfaces with a systole roughly equal to~$\log(g)$.

In higher dimension, M.~Gromov \cite{gro83}, \cite{gro96} related the systolic volume~$\sigma({\bf a})$ of a homology class ${\bf a} \in H_m(X;\Z)$ to its simplicial volume~$\Vert {\bf a} \Vert_\Delta$ through the following lower bound
\begin{equation} \label{eq:sigma-lambda}
\sigma({\bf a}) \geq \lambda_m \, \frac{\Vert {\bf a} \Vert_\Delta}{(\log(2+ \Vert {\bf a} \Vert_\Delta))^m}
\end{equation}
where $\lambda_m$ is a positive constant depending only on~$m$.
In particular, for every ${\bf a} \in H_m(X;\Z)$ with nonzero simplicial volume, we have
\begin{equation} \label{eq:ka}
\sigma(k \, {\bf a}) \geq \lambda \, \frac{k}{(\log k)^m}
\end{equation}
where $\lambda=\lambda({\bf a}) >0$.

\medskip

In a different direction, one can ask for an asymptotic upper bound on~$\sigma(k \, {\bf a})$.
This problem was considered in~\cite{BB05}, where a sublinear upper bound in~$k$ was established, and in~\cite{BB15}, where the upper bound was improved.

\medskip

Using different techniques, we obtain an upper bound showing that the lower bound~\eqref{eq:ka} is roughly optimal in~$k$, which positively answers a conjecture of~\cite{BB15}.

\begin{theorem}
Let $X$ be a path-connected topological space.
Then for every homology class ${\bf a} \in H_m(X;\Z)$, there exists a constant $C=C({\bf a}) >0$ such that for every integer $k \geq 2$, we have
\[
\sigma(k \, {\bf a}) \leq C \, \frac{k}{(\log k)^m}.
\]
\end{theorem}

This estimate allows us to define the systolic semi-norm in real homology, see \S\ref{subsec:systolic}.

\medskip

Articles about minimal volume entropy closely related to our paper include \cite{B93}, \cite{B06}, \cite{B06b}, \cite{BCG91}, \cite{BCG95}, \cite{Brunnbauer08}, \cite{gro82}, \cite{merlin}, \cite{sab}, \cite{sab17}, \cite{Sambusetti99}, \cite{Sambusetti00}. \\

This article is organized as follows.
In Section~\ref{sec:connectedsum}, we establish lower and upper bounds on the minimal volume entropy of the connected sum of closed manifolds, and derive that the functional~$\Vert \cdot \Vert_E$ is a semi-norm in real homology.
Functorial properties of the volume entropy semi-norm are presented in Section~\ref{sec:functorial}.
In Section~\ref{sec:E<S}, we show that the volume entropy semi-norm of a homology class is bounded from above and below by its simplicial volume, up to some multiplicative constants depending only on the degree of the homology class.
Therefore, the volume entropy semi-norm and the simplicial volume are equivalent homology semi-norms.
Our approach for the upper bound relies on a geometrization of the simplicial volume and the universal realization of homology classes established by A.~Gaifullin \cite{gaifullin08a}, \cite{gaifullin08b} regarding Steenrod's problem.
More than the result about the universal realization of homology classes, we will need to retrieve combinatorial features of the construction to apply our argument leading to an upper bound on the volume entropy semi-norm of a homology class.
The reverse inequality is obtained through the use of bounded cohomology by adapting M.~Gromov's chain diffusion technique.
In the last section, we bound from above the systolic volume of the multiple of a given homology class.
The proof relies on topological properties of the universal realizators in homology used in the previous section and on systolic estimates in geometric group theory. 
We also introduce the systolic semi-norm in real homology. \\

\noindent {\it Acknowledgment.}
The second author would like to thank the Fields Institute and the Department of Mathematics at the University of Toronto for their hospitality while this work was completed.

\section{Entropy of connected sums} \label{sec:connectedsum}

In this section, we first establish an additive formula of the functional~$\Omega$ for the bouquet of polyhedra.
We also obtain lower and upper bounds on the minimal volume entropy of the connected sum of two closed manifolds, and derive that the functional~$\Vert \cdot \Vert_E$ is a semi-norm in real homology.
Finally, we present a couple of applications of these estimates.

\subsection{Preliminaries}

\mbox{ }

\medskip



Let us first recall the definition of a pseudomanifold.

\begin{definition} \label{def:pseudomanifold}
A connected closed $m$-dimensional pseudomanifold is a finite simplicial complex~$M$ such that
\begin{enumerate}
\item every simplex of~$M$ is a face of some $m$-simplex of~$M$; \label{p1}
\item every $(m-1)$-simplex of~$M$ is the face of exactly two $m$-simplices of~$M$; \label{p2}
\item given two $m$-simplices $s$ and $s'$ of~$M$, there exists a finite sequence $s=s_1, s_2, \dots, s_n=s'$ of $m$-simplices of~$M$ such that $s_i$ and $s_{i+1}$ have an $(m-1)$-face in common. \label{p3}
\end{enumerate}

The $m^{\text{th}}$ homology group~$H_m(M;\Z)$ of a connected closed $m$-dimensional pseudomanifold is either isomorphic to~$\Z$ or trivial, see~\cite{spa}.
In the former case, we say that the pseudomanifold~$M$ is orientable. 
\end{definition}

\medskip

Consider a finite simplicial complex~$X$ with a piecewise Riemannian metric~$g$.
Denote by~$\rho$ the distance induced by~$g$ on~$X$ and on all the coverings of~$X$.
Let $H \lhd G$ where $G = \pi_1(X)$.
The quotient group~$G/H$ acts by isometries on the $H$-covering~$X_H$.
Furthermore, the action of~$G/H$ on~$X_H$ is proper, discontinuous, without any fixed point.
Fix $q \in X_H$.
The orbit of~$q$ under the action of~$G/H$ on~$X_H$ is denoted by~$q \cdot (G/H)$.
Let also
\[
B_H(t,q;g) = \{ x \in X_H \mid \rho(q,x) \leq t \}
\]
be the ball of radius~$t$ centered at~$q$ in~$X_H$.

\medskip

The volume entropy of~$X$ relative to~$H$ is equal to the exponential growth rate of the number of points in the orbit of~$q$ under~$G/H$, as stated in the following classical result, see~\cite[Lemma~2.3]{sab} for instance.

\begin{proposition} \label{prop:folk}
With the previous notations, 
\begin{equation}\label{eq:entropie.discrete}
\ent_H(X, g) = \mathop{\lim}\limits_{t \rightarrow \infty} \frac{1}{t} \, \log \vert B_H(t, q; g)\cap q \cdot (G/H)\vert.
\end{equation}
\end{proposition}

\subsection{Minimal volume entropy of a bouquet of polyhedra}

\mbox{ } 

\medskip

Let us recall a few results established in~\cite{B93}, see also~\cite[Lemma~3.5]{sab}.

\begin{definition} \label{def:monotone}
A simplicial map $f:X \to Y$ between two $m$-dimensional simplicial complexes is \emph{$m$-monotone} if for every point~$y$ in the interior of an $m$-simplex of~$Y$, the preimage~$f^{-1}(y)$ is connected (and so is a point).
\end{definition}

We will need the following comparison principle proved in~\cite[\S2]{B93}. 

\begin{proposition} \label{prop:comparison}
For $i=1,2$, let $X_i$ be an $m$-dimensional simplicial complex and \mbox{$\phi_i:\pi_1(X_i) \to G$} be an epimorphism.
Suppose that there exists an $m$-monotone map $f: X_1 \to X_2$ such that $\phi_1 = \phi_2 \circ f_*$.
Then
\[
\Omega_{H_1}(X_1) \leq \Omega_{H_2}(X_2)
\]
where $H_i = \ker \phi_i$.
\end{proposition}

Actually, Proposition~\ref{prop:comparison} is a straighforward consequence of the following result proved in~\cite[\S2]{B93} and~\cite[Lemme~3.1]{B06}, which will also be used in the sequel.

\begin{lemma} \label{lem:nonexpanding}
Let $f:X \to Y$ be an $m$-monotone map between two $m$-dimensional simplicial complexes.
Then for every polyhedral Riemannian metric~$g$ on~$Y$ and every $\e >0$, there exists a polyhedral Riemannian metric~$g_\e$ on~$X$ with
\[
\vol(X,g_\e) \leq \vol(Y,g) + \e
\]
such that $f$ is nonexpanding.
\end{lemma}

Let $G$ be a finitely presented group.
For every subgroup~$H$ of~$G$, denote by $\langle H \rangle$ the normal closure of~$H$ in~$G$.

\medskip

The following result provides a formula for the minimal volume entropy of the bouquet of two simplicial polyhedra.

\begin{theorem}\label{th:strong.additif}
For $i=1,2$, let $K_i$ be a connected $m$-dimensional simplicial polyhedron and $H_i \lhd \pi_1(K_i)$ be a normal subgroup.
Then
\begin{equation}\label{eq:additivite1}
\Omega_{\langle H_1 \ast H_2\rangle } (K_1 \vee K_2) = \Omega_{H_1} (K_1) + \Omega_{H_2} (K_2).
\end{equation} 
\end{theorem}

\begin{proof}
First we prove the inequality
\begin{equation}\label{eq:borne.inf}
\Omega_{H_1} (K_1) + \Omega_{H_2} (K_2) \leq  \Omega_{\langle H_1 \ast H_2\rangle } (K_1 \vee K_2).
\end{equation}
Let $X = K_1 \vee K_2$.
By van Kampen's theorem~\cite[\S1.2]{hatcher}, we have
\[
\pi_1(X) =  \pi_1(K_1) \ast \pi_1(K_2).
\]
Let $g$ be a polyhedral Riemannian metric on $X$ and $g_i$ be its restriction to~$K_i$ for~$i=1,2$.
Let $\widehat{K_i}$ and $\widehat{X}$ be the normal covers corresponding to the nomal subgroups $H_i \lhd \pi_1(K_i)$
and $\langle H_1*H_2\rangle \lhd \pi_1(X)$, with the lifted metrics~$\widehat{g_i}$ and~$\widehat{g}$.
Observe that the natural inclusions $\widehat{K_i} \subseteq \widehat{X}$ are isometric.
This implies 
$$
\ent_{H_i}(K_i, g_i) \leq \ent_{\langle H_1 \ast H_2\rangle }(X, g)
$$
Thus, for every $i=1,2$, 
$$
\Omega_{H_i} (K_i) \leq \ent_{H_i}(K_i, g_i)^m \, \vol(K_i, g_i) \leq \ent_{\langle H_1 \ast H_2\rangle }(X, g)^m \, \vol(K_i, g_i)
$$
Adding the two inequalities so-obtained for $i=1,2$, and using the relation 
\[
\vol(X, g) = \vol(K_1, g_1) + \vol(K_2, g_2)
\]
we finally derive
$$
\Omega_{H_1} (K_1) + \Omega_{H_2} (K_2) \leq \Omega_{\langle H_1 \ast H_2\rangle } (X, g) 
$$
for every polyhedral metric $g$ on $X$.
Hence the inequality~\eqref{eq:borne.inf}.

\medskip

Now, let us prove the reverse inequality

\begin{equation}\label{eq:borne.sup}
\Omega_{\langle H_1 \ast H_2\rangle } (K_1 \vee K_2) \leq \Omega_{H_1} (K_1) + \Omega_{H_2} (K_2).
\end{equation}

We proceed in two steps.
Without loss of generality, we can assume that the two sub-complexes~$K_1$ and~$K_2$ of~$K$ are glued at a common vertex.
Let $p_i \in K_i$ be a vertex such that 
$$
 K_1 \vee K_2 = K_1 \mathop{\cup}\limits_{p_1=p_2} K_2.
$$
Define the $m$-dimensional simplicial complex
\begin{equation}\label{eq:complexe.K}
K = K_1 \mathop{\cup}\limits_{p_1= \{1\}} [1, 2] \mathop{\cup}\limits_{\{2\} = p_2}  K_2.
\end{equation}

For the first step, let us show that 
\[
\Omega_{\langle H_1 \ast H_2\rangle } (K) =  \Omega_{\langle H_1 \ast H_2\rangle } (K_1 \vee K_2).
\]
Contracting the interval~$[1, 2]$ in~$K$  to a point gives rise to an $m$-monotone simplicial map 
\[
K \longrightarrow  K_1 \vee K_2
\]
inducing a $\pi_1$-isomorphism.
By Proposition~\ref{prop:comparison}, we derive
\begin{equation}\label{eq:borne.inf.1}
 \Omega_{\langle H_1 \ast H_2\rangle } (K) \leq  \Omega_{\langle H_1 \ast H_2\rangle } (K_1 \vee K_2).
\end{equation}

For the reverse inequality, let~$\te_i$ be a triangulation of~$K_i$ for every $i=1, 2$. 
Denote by~${\rm St}(p_i)$ the open star of~$p_i$ for the triangulation~$\theta_i$.
Let $\theta'_i$ be the triangulation of~$K_i$ which agrees with~$\theta_i$ in $K_i \setminus {\rm St}(p_i)$ and with the semi-barycentric triangulation of~${\rm St}(p_i)$ in~${\rm St}(p_i)$ (obtained by adding a vertex at the barycenter of the $m$-simplices of~${\rm St}(p_i)$).
The bouquet~$K_1 \vee K_2$ is endowed with the triangulations given by~$\theta'_1$ and~$\theta'_2$.
The complex $K=K_1 \cup [1,2] \cup K_2$ is endowed with the triangulation given by~$\theta_1$, $\theta_2$ and the barycentric subdivision of~$[1,2]$ into $I_1=[1,\frac{3}{2}]$ and~$I_2=[\frac{3}{2},2]$.

Consider the simplicial map
\[
f: K_1 \vee K_2 \to K
\]
which agrees with the identity map on~$K_i \setminus {\rm St}(p_i)$, and takes~$p_i$ to the midpoint of~$[1,2]$ and all the vertices of~$\theta_i'$ corresponding to the barycenter of the $m$-simplices of~${\rm St}(p_i)$ for the triangulation~$\theta_i$ to~$i$.
By construction, the map~$f$ is $m$-monotone and induces a $\pi_1$-isomorphism.

\forget

Let $\mbox{St}(p_i)$, $i=1, 2$ be the stars in 
$\te_i$ of the vertexes $p_i$.
Denote by $\te_i'$ the semi-barycentric triangulation of $\mbox{St}(p_i)$ vhich coincide vith  $\te_i $
in $K_i \setminus \mbox{St}(p_i)$. For this semi-barycentric subdivision of a simplex we proceed as follows.
Let $\Delta^m = (v_0, v_1, \dots , v_m)$ be the standart $m$-simplexe. Denote by $b_{0i_1\dots i_k}$ the barycenter
of a $k$-face $(v_0, v_{i_1}, \dots , v_{i_k})$ always containing the vertex $(v_0)$. Semi-baricentric 
subdivision $\Delta^m$ is the natural riangulation of $\Delta^m$ such that the mapping keeping all vertexes
fixed and contracting all concidered barycenters in $b = b_{01\dots m}$ extends to the simplicial map
$$
p: (v_0, v_1, \dots , v_m) \longrightarrow (v_0, b) \mathop{\cup}\limits_{b}
(b, v_1, \dots , v_m)
$$
of simplicial polyhedra. Note that $p$ an $m$-monotone map by construction. By applying this construction to all
simplicies of  $\mbox{St}(p_i)$, $i=1, 2$ where we always choos $v_0 = p_i$ we obtain desired
subdivisions $\te_i'$, $i=1, 2$. This leeds to the natural monoton simplicial map
$$
P:  K_1 \vee K_2 \longrightarrow K.
$$ 

\forgotten

The inequality obtained by applying Proposition~\ref{prop:comparison} to~$f$, combined with the inequality~\eqref{eq:borne.inf.1}, yields the relation
\begin{equation}\label{eq:equality}
 \Omega_{\langle H_1 \ast H_2\rangle } (K) =  \Omega_{\langle H_1 \ast H_2\rangle } (K_1 \vee K_2).
\end{equation}

For the second step, we need to show that
\begin{equation}\label{eq:borne.sup.bis}
\Omega_{\langle H_1 \ast H_2\rangle } (K) \leq \Omega_{H_1} (K_1) + \Omega_{H_2} (K_2).
\end{equation}
Fix $\b_i> \Omega_{H_i}(K_i)^\frac{1}{m}$.
By definition, there exists a metric~$h_i$ on~$K_i$ such that
\begin{equation}\label{eq:1}
\ent_{H_i}(K_i, h_i)^m \, \vol(K_i, h_i)  < \b_i^m.
\end{equation}
By scale invariance, this inequality holds for every homothetic metric~$\l_i^2h_i$ with $\l_i > 0$.
Choose the factors $\l_1$ and~$\l_2$ so that 
\begin{equation} \label{eq:2}
\ent_{H_1}(K_1, \l_1^2h_1) = \ent_{H_2}(K_2, \l_2^2h_2)
\end{equation}
and
\begin{equation} \label{eq:2b}
\vol(K_1, \l_1^2h_1) + \vol(K_2, \l_2^2h_2) = 1.
\end{equation}
Let $\a = \ent_{H_1}(K_1, \l_1^2h_1) = \ent_{H_2}(K_2, \l_2^2h_2)$.
The relations~\eqref{eq:1} and~\eqref{eq:2} combined with~\eqref{eq:2b} show that
\begin{equation}\label{eq:3}
\a^m < \b_1^m + \b_2^m.
\end{equation}

Consider the metric~$g_d$ on~$K$ which is defined on the three parts of~$K$ given by~\eqref{eq:complexe.K} as follows
\begin{equation}\label{eq:metrique.sur.K}
g_d =
\begin{cases}
\l_1^2h_1 & \text{on } K_1 \\
4d^2 dx^2  & \text{on } [p_1,p_2] \\
\l_2^2h_2 & \text{on } K_2
\end{cases}
\end{equation}
where $x$ is the coordinate on~$[p_1,p_2]=[1, 2]$ and $d>0$ is a parameter.
By construction, we have $\length_{g_d}([p_1,p_2]) = 2d$ and $\vol(K,g_d)  = 1$,
where the second equality comes from~\eqref{eq:2b}.

\medskip

We will need the following result.

\begin{lemma} \label{lem:ae}
Let $\e >0$.
For $d$ large enough, we have
$$
\ent_{\langle H_1 \ast H_2\rangle } (K, g_d) < \a + \e. 
$$
\end{lemma}
\begin{proof}
Let $\widehat{K_i}$ and $\widehat{K}$ be the normal covers corresponding to the normal subgroups $H_i \lhd \pi_1(K_i)$ and $\langle H_1 \ast H_2\rangle \lhd \pi_1(K)$.
The cover~$\widehat{K}$ of $K=K_1 \cup [p_1,p_2] \cup K_2$ can be described as follows
\begin{enumerate}
\item the cover~$\widehat{K}$ decomposes into the union of the lifts of the subsets~$K_1$ and~$K_2$ of~$K$, also called \emph{leaves} of~$\widehat{K}$, and the lifts of~$[p_1,p_2]$; \label{description1}
\item every lift of~$[p_1,p_2]$ in~$\widehat{K}$ is adjacent to two leaves homeomorphic to~$\widehat{K_1}$ and~$\widehat{K_2}$;
\item removing a lift of~$[p_1,p_2]$ from~$\widehat{K}$ separates the cover into two connected components.  \label{description4}
\end{enumerate}

The quotient action of $G = \pi_1(K)/\langle H_1 \ast H_2\rangle$ on~$\widehat{K}$ where 
\[
\pi_1(K) = \pi_1(K_1 \vee K_2) = \pi_1(K_1) \ast \pi_1(K_2)
\]
decomposes into 
\[
G = G_1 \ast G_2,
\]
where $G_i=\pi_1(K_i)/H_i$ (this relation is left to the reader as an exercice in group theory).
With this decomposition, the action of~$G$ on~$\widehat{K}$ can be described as follows.
Let $F \simeq \widehat{K}_i$ be a leaf of~$\widehat{K}$.
The subgroup $G_i=\pi_1(K_i)/H_i$ of~$G$ acts on~$F \subseteq \widehat{K}$.
For every lift~$q_i$ of~$p_i$ in~$F$, the orbit $q_i \cdot G_i$ of~$q_i$ in~$F$ is composed of all the lifts of~$p_i$ lying in~$F$ under the cover $\widehat{K} \to K$.

\medskip

Denote by $\widehat{\r}_i$ the distance on~$\widehat{K_i}$ induced by~$\l_i^2h_i$ and denote by~$\widehat{g}_d$ the metric on~$\widehat{K}$ induced by~$g_d$, see~\eqref{eq:metrique.sur.K}.
Let $[q_1,q_2]$ be a lift of~$[p_1,p_2]$ in~$\widehat{K}$ and $q$ be the midpoint of~$[q_1,q_2]$.
In view of~\eqref{eq:entropie.discrete}, the desired bound on~$\ent_{\langle H_1 \ast H_2\rangle } (K, g_d)$ will follow from a bound on~$v(t; d) = \vert V(t; d) \vert$, where
\begin{equation}\label{eq:V(t; d)}
V(t; d) =  B_{\langle H_1 \ast H_2\rangle }(t, q; g_d)\cap (q \cdot G).
\end{equation}

\medskip

By the normal form theorem for free product of groups, see~\cite{Kurosh60}, every element $\gamma \in G \simeq G_1 \ast G_2$ can be uniquely written in normal form as
\begin{equation}\label{eq:a}
\gamma = \gamma_1 \gamma_2 \dots \gamma_l
\end{equation}
where $\gamma_s \in G_{i_s}$ is nontrivial and $i_s \neq i_{s+1}$ 
for every $s \in \{1, \dots , l-1\}$. 
The \emph{length} $l(\gamma)$ of~$\gamma$ is the number~$l$ of elements in the decomposition~\eqref{eq:a}.
It follows from the description of the cover~$\widehat{K}$, see \eqref{description1}-\eqref{description4}, and of the action of~$G_i$ on every leaf~$F \simeq \widehat{K}_i$ of~$\widehat{K}$ that every path from~$q$ to~$q \cdot \gamma$, where $\gamma$ is of length~$l$, passes through the points
\[
q,\dots, q_{i_s} \cdot (\gamma_{s-1} \dots \gamma_1), q_{i_s} \cdot (\gamma_s \dots \gamma_1), q_{i_{s+1}} \cdot (\gamma_{i_{s}} \dots \gamma_1), \dots, q \cdot \gamma
\]
where $s$ runs over $\{ 1,\dots,l-1 \}$.
Indeed, removing any of these points between $q$ and~$q \cdot \gamma$ from~$\widehat{K}$ disconnects $q$ and~$q \cdot \gamma$.
Since $G$ acts by isometries on~$\widehat{K}$, the $\widehat{g}_d$-distance between $q_{i_s} \cdot (\gamma_1 \dots \gamma_{s-1})$ and $q_{i_s} \cdot (\gamma_1 \dots \gamma_s)$ is equal to ${\rm dist}_{\widehat{g}_d}(q_{i_s},q_{i_s} \cdot \gamma_s)$.
Since the restriction of the distance~${\rm dist}_{\widehat{g}_d}$ to a leaf $F \simeq \widehat{K_i}$ of~$\widehat{K}$ agrees with the distance~$\widehat{\r}_i$ on~$\widehat{M_i}$ and since the action of~$G_i$ on~$F$ as a subgroup of~$G$ coincides with its action on~$\widehat{K_i}$ under the identification $F \simeq \widehat{K_i}$, we have
\[
{\rm dist}_{\widehat{g}_d}(q_{i_s},q_{i_s} \cdot \gamma_s) = \widehat{\rho}_i(p_{i_s},p_{i_s} \cdot \gamma_s).
\]
Thus,
$$
{\rm dist}_{\widehat{g}_d}(q, q \cdot \gamma)  =  d + \widehat{\r}_{i_1}(p_{i_1}, p_{i_1} \cdot \gamma_1) + 2d + 
\widehat{\r}_{i_2}(p_{i_2}, p_{i_2} \cdot \gamma_2) + 2d + \dots + \widehat{\r}_{i_l}(p_{i_l}, p_{i_l} \cdot \gamma_l) + d.
$$
Hence,
\begin{equation}\label{eq:longueur}
{\rm dist}_{\widehat{g}_d}(q, q \cdot \gamma)  =  2dl + \mathop{\sum}\limits_{s=1}^l  \widehat{\r}_{i_s}(p_{i_s}, p_{i_s} \cdot \gamma_s).
\end{equation}

\medskip

To estimate the exponential growth rate of the orbit of~$G$ in~$\widehat{K}$, it will be useful to decompose~$G$ by the filtration induced by the length on~$G$.
Under this filtration, the group~$G$ decomposes into the disjoint union
$$
G = \mathop{\cup}\limits_{l=1}^{\infty}G^{(l)}
$$
where $G^{(l)}$ is formed of the elements of~$G$ of length~$l$. 
We derive from~\eqref{eq:V(t; d)} that
\begin{equation}\label{eq:UVl}
V(t; d) = \mathop{\cup}\limits_{l \geq 1} V^{(l)}(t; d)
\end{equation}
where
\[
V^{(l)}(t; d) =  B_{\langle H_1 \ast H_2\rangle }(t, q; g_d)\cap (q \cdot G^{(l)}).
\]
Since the union~\eqref{eq:UVl} is disjoint, we can write
\[
v(t; d) = \mathop{\sum}\limits_{l\geq 1}v^{(l)}(t; d)
\]
where $v^{(l)}(t; d) = \vert V^{(l)}(t; d) \vert$. 

\medskip

Suppose $q \cdot \gamma \in V^{(l)}(t; d)$.
Let $t_s$ be the smallest integer greater or equal to $\widehat{\r}_{i_s}(p_{i_s}, p_{i_s} \cdot \gamma_s)$.
Then $p_{i_s} \cdot \gamma_s \in V_{i_s}(t_s)$, where
\[
V_i(t) = B_{H_i}(t,p_i;\widehat{\r}_i) \cap \left( p_i \cdot G_i \right).
\]
Furthermore, 
\[
t_s  < \widehat{\r}_{i_s}(p_{i_s}, p_{i_s} \cdot \gamma_s) + 1.
\]
By~\eqref{eq:longueur}, this inequality leads to
\begin{align}
\sum_{s=1}^l t_s & < \sum_{s=1}^l \widehat{\r}_{i_s}(p_{i_s}, p_{i_s} \cdot \gamma_s) + l \nonumber \\
 & < t-2dl+l = t-(2d-1)l. \label{eq:sumts}
\end{align}
Therefore, every element $\gamma \in G$ with $q \cdot \gamma \in V(t;d)$ decomposes into a product $\gamma=\gamma_1 \dots \gamma_l$ with $\gamma_s \in G_{i_s}$, see~\eqref{eq:a}, such that $p_{i_s} \cdot \gamma_s \in V_{i_s}(t_s)$, where the integers~$t_s$ defined from~$\gamma_s$ satisfy~\eqref{eq:sumts}.
The number of elements $\gamma \in G$ of length~$l$ with $q \cdot \gamma \in V(t;d)$ and given integers~$t_s$ satisfying~\eqref{eq:sumts} is at most
$$
\vert V_{i_1}(t_1) \vert \cdot \vert  V_{i_2}(t_2) \vert \cdots \vert  V_{i_l}(t_l)\vert.
$$

By definition of~$\alpha$, see below~\eqref{eq:2b}, the exponential growth rate of~$\vert V_{i_s}(t)\vert$ agrees with~$\alpha$.
Thus, 
for every $\a_1 > \a$ (arbitrarily close to~$\alpha$, which will be specified afterwards), there exists $t_0>0$ such that $\vert V_{i_s}(t)\vert < e^{\a_1t}$ for every $t > t_0$. 

Let $I$ be the subset of $L = \{1, \dots, l \}$ given by  
\[ 
I= \{ s \in L \mid t_s \leq t_0 \}.
\]
Let $C = \max\{V_1(t_0), V_2(t_0)\}$.
For every $s \in L$, we have
\[
\begin{array}{ll}
\vert V_{i_s}(t_s) \vert \leq C & \text{if } s \in I \\
\vert V_{i_s}(t_s) \vert \leq e^{\a_1t_s} & \text{if } s \notin I.
\end{array}
\]
These estimates yield an upper bound on the product
\begin{equation}\label{eq:produit}
\vert V_{i_1}(t_1) \vert \cdot \vert  V_{i_2}(t_2) \vert \cdots \vert  V_{i_l}(t_l)\vert \leq
C^{\vert I \vert} e^{\a_1 \left( \mathop{\sum}\limits_{s \notin I}t_s \right)} \leq 
C^l  e^{\a_1(t-(2d-1)l)}
\end{equation}
where the last inequality follows from $\vert I \vert \leq l$ and the bound~\eqref{eq:sumts}.

Now, the number of $l$-uplets $\t = (t_1, t_2, \dots , t_l)$ with nonnegative integral coordinates satisfying~\eqref{eq:sumts} is bounded by 
\[
\frac{[t-(2d-1)l]^l}{l!}.
\]
Combined with~\eqref{eq:produit}, this leads to
\[
v^{(l)}(t; d) \leq 
\mathop{\sum}\limits_{\t} \vert V_{i_1}(t_1) \vert \cdot \vert  V_{i_2}(t_2) \vert \cdots \vert  V_{i_l}(t_l)\vert \leq C^l  e^{\a_1(t-(2d-1)l)}\frac{[t-(2d-1)l]^l}{l!}
\]
where $\t$ runs over all $l$-uplets satisfying~\eqref{eq:sumts}.
For $d > \frac{1}{2}$, we have $t-(2d-1)l \leq t$ and so
$$
v^{(l)}(t; d) \leq C^l  e^{\a_1t} e^{-\a_1(2d-1)l}\frac{t^l}{l!} = 
 \frac{e^{\a_1t}}{l!}  \left( \frac{Cd}{e^{\a_1(2d-1)}} \right)^l \left(\frac{t}{d}\right)^l. 
$$
For $d$ large enough, we have
\begin{equation}\label{eq:d.grand.1}
\frac{Cd}{e^{\a_1(2d-1)}} < 1.
\end{equation}
Thus,
$$
v^{(l)}(t; d) \leq   \frac{e^{\a_1t}}{l!} \left(\frac{t}{d}\right)^l. 
$$
Therefore,
$$
v(t; d) = \mathop{\sum}\limits_{l\geq 1}v^{(l)}(t; d) \leq  \mathop{\sum}\limits_{l\geq1} \frac{e^{\a_1t}}{l!} \left(\frac{t}{d}\right)^l = e^{\a_1 t + \frac{t}{d}}.
$$
Hence,
\begin{equation}\label{eq:d.grand.2}
\mathop{\lim}\limits_{t \rightarrow \infty}\frac{\log v(t; d)}{t} \leq 
\a_1 + \frac{1}{d}.
\end{equation}
For $\a_1 < \a + \e$ and $d > \frac{1}{\e}$ large enough so that the inequality~\eqref{eq:d.grand.1} is satisfied, we deduce from~\eqref{eq:d.grand.2} that 
$$
\ent_{\langle H_1 \ast H_2\rangle } (K, g_d) < \a + 2\e
$$
which finishes the proof of the lemma.
\end{proof}

Let us resume the proof of Theorem~\ref{th:strong.additif}.
Since $\beta_i^m$ can be arbitrarily close to~$\Omega_{H_i}(K_i)$, the inequality~\eqref{eq:3} combined with Lemma~\ref{lem:ae} leads to
\[
\Omega_{\langle H_1 \ast H_2\rangle } (K) \leq \Omega_{H_1} (K_1) + \Omega_{H_2} (K_2).
\]
Along with~\eqref{eq:equality}, this inequality yields the desired result.
\end{proof}

\subsection{Upper bound on the minimal volume entropy of connected sums}

\mbox{ } 

\medskip

The following result compares the minimal volume entropy of the connected sum of two manifolds with the minimal volume entropy of the two manifolds.

\begin{theorem}\label{th:sou.additif}
For $i=1,2$, let $M_i$ be a connected closed pseudomanifold of dimension $m \geq 3$ and $H_i \lhd \pi_1(M_i)$ be a normal subgroup.
Then
\begin{equation}\label{eq:semiadditivite1}
\Omega_{\langle H_1 \ast H_2\rangle } (M_1 \sharp M_2) \leq \Omega_{H_1} (M_1) + \Omega_{H_2} (M_2).
\end{equation} 
\end{theorem}

\begin{remark} \label{rem:sou.additif}
In dimension~$2$, the result remains valid by replacing the normal subgroup ${\langle H_1 \ast H_2\rangle}$ with $f^{-1}({\langle H_1 \ast H_2\rangle })$, where
$
f: M_1 \sharp M_2 \longrightarrow M_1 \vee M_2
$
is the natural projection. 
\forget
In dimension 1 the theorem is trivial : all termes are equal to 0.
In dimension 2 pseudo-manifolds are surfaces pinched in few points. The result remains valid but the
normal sub-group ${\langle H_1 \ast H_2\rangle }$ has to be understood as
$\pi_1(p)^{-1}({\langle H_1 \ast H_2\rangle })$ where
$$
p: M_1 \sharp M_2 \longrightarrow M_1 \vee M_2
$$
is the natural projection. Note also that to consider only surfaces than pseudomanifolds is more natural
in dimension 2 .
\forgotten
\end{remark}

\begin{remark}
Inequality~\eqref{eq:semiadditivite1} is the analogue for the volume entropy of a similar bound holding for the systolic volume, see~\cite[Proposition~3.6]{BB15}.
\end{remark}


\begin{proof}[Proof of Theorem~\ref{th:sou.additif}]
Consider the natural $m$-monotone map 
$$
f: M_1 \sharp M_2 \to M_1 \vee M_2
$$
obtained by collapsing the attaching sphere to a point (in order to get a simplicial map, we may have to take two barycenter subdivisions of~$M_1$ and~$M_2$).
Since $m \geq 3$, the induced homomorphism $f_*:\pi_1(M_1 \sharp M_2) \to \pi_1(M_1 \vee M_2)$ is an isomorphism.
The comparison principle, see Proposition~\ref{prop:comparison}, and Theorem~\ref{th:strong.additif} yield
$$
\Omega_{\langle H_1 \ast H_2\rangle } (M_1 \sharp M_2) \leq \Omega_{\langle H_1 \ast H_2\rangle } (M_1 \vee  M_2) = \Omega_{H_1} (M_1) + \Omega_{H_2} (M_2).
$$ 
\end{proof}

\begin{corollary} \label{coro:seminorm}
Let $X$ be a path-connected topological space.
Then for every ${\bf a}, {\bf b} \in H_m(X;\Z)$, we have
\[
\Omega({\bf a} + {\bf b}) \leq \Omega({\bf a}) + \Omega({\bf b}).
\]
In particular, the quantity $\Omega({\bf a} - {\bf b})$ defines a pseudo-distance between ${\bf a}$ and ${\bf b}$ in~$H_m(X;\Z)$, and the functional~$\Vert \cdot \Vert_E$ is a semi-norm on~$H_m(X;\R)$.
\end{corollary}
\begin{proof}
Let ${\bf a}_1, {\bf a}_2 \in H_m(X;\R)$.
Fix $\varepsilon >0$.
There exists a map $f_i:M_i \to X$ from an orientable connected closed $m$-pseudomanifold~$M_i$ representing~${\bf a}_i$ for $i=1,2$ such that 
\[
\Omega_{{\rm ker} (f_i)_*}(M_i) \leq \Omega({\bf a}_i) + \varepsilon.
\]
Let $M = M_1 \sharp M_2$.
Consider the natural map $f=f_1 \vee f_2:M \to X$ obtained from~$f_1$ and~$f_2$ by first collapsing the attaching map to a point.
Note that 
\[
\ker f_* = \langle \ker (f_1)_* * \ker (f_2)_* \rangle.
\]
Furthermore, by Theorem~\ref{th:sou.additif}, we have
\begin{align*}
\Omega_{{\rm ker} f_*}(M) & \leq \Omega_{{\rm ker} (f_1)_*}(M_1) + \Omega_{{\rm ker} (f_2)_*}(M_2) \\
 & \leq \Omega({\bf a}_1) + \Omega({\bf a}_2) + 2 \varepsilon
\end{align*}
Hence, $\Omega({\bf a}_1 + {\bf a}_2) \leq \Omega({\bf a}_1) + \Omega({\bf a}_2)$.

Replacing ${\bf a}_i$ with~$k \, {\bf a}_i$ in the previous inequality, dividing by~$k$ and letting $k$ go to infinity, we obtain 
\[
\Vert {\bf a}_1 + {\bf a}_2 \Vert_E \leq \Vert {\bf a}_1 \Vert_E + \Vert {\bf a}_2 \Vert_E.
\]
Since $\Vert \cdot \Vert_E$ is clearly homogeneous by the stabilization process, see~\eqref{eq:semi-norme-entropie}, the functional~$\Vert \cdot \Vert_E$ is a semi-norm.
\end{proof}

\subsection{Lower bound on the minimal volume entropy of connected sums}

\mbox{ } 

\medskip

In a different direction, taking the connected sum with an orientable pseudomanifold does not decrease the minimal volume entropy as the following result shows.

\begin{theorem}\label{theo:stabilisation}
For $i=1,2$, let $M_i$ be a connected closed pseudomanifold of dimension $m \geq 3$ and $H_i \lhd \pi_1(M_i)$ be a normal subgroup. Let $H \lhd \pi_1(M_1) \ast \pi_1(M_2)$ be a normal subgroup such that the natural inclusion
$\pi_1(M_1) <  \pi_1(M_1) \ast \pi_1(M_2)$ induces an inclusion 
\begin{equation}\label{eq:group.inclusion}
\pi_1(M_1)/H_1 < (\pi_1(M_1) \ast \pi_1(M_2))/H
\end{equation} 
Suppose $M_2$ is orientable.
Then
\begin{equation}\label{eq:stabilisation}
\Omega_{H_1} (M_1) \leq \Omega_H (M_1 \sharp M_2).
\end{equation} 
\end{theorem}

Combined with Theorem~\ref{th:sou.additif}, we obtain
\begin{corollary}\label{cor:stabilisation1}
For $i=1,2$, let $M_i$ be a connected closed pseudomanifold of dimension $m \geq 3$ and $H_i \lhd \pi_1(M_i)$ be a normal subgroup.
Suppose $M_2$ is orientable and $\Omega_{H_2} (M_2) = 0$.
Then
$$
\Omega_{\langle H_1 \ast H_2\rangle } (M_1 \sharp M_2) = \Omega_{H_1} (M_1).
$$ 
\end{corollary}

In order to prove Theorem~\ref{theo:stabilisation}, we first establish the following result. 
For a $CW$-complex~$X$, denote by~$X(k)$ its $k$-skeleton.

\begin{proposition} \label{prop:bouquet}
Let $M$ be an orientable connected closed pseudomanifold of dimension~\mbox{$m \geq 3$}.
Suppose that
\begin{equation}\label{eq:cell.decomposition}
M = D^m \mathop{\cup}_{\phi} M(m-1)
\end{equation}
is a cell decomposition with a single $m$-cell.
Then the space 
\[
M \mathop{\cup}\limits_{M(m-2)} \Cone(M(m-2))
\]
obtained by gluing the cone~$\Cone(M(m-2))$ over~$M(m-2)$ to~$M$ along~$M(m-2)$ is homotopy equivalent to a finite bouquet of spheres
\[
M \mathop{\cup}\limits_{M(m-2)} \Cone(M(m-2)) \simeq \mathop{\bigvee}\limits_s S^{m-1}_s \bigvee S^m.
\]
\end{proposition}

\begin{proof}
We have
\[
M \mathop{\cup}\limits_{M(m-2)} \Cone(M(m-2)) \simeq M/M(m-2) \simeq \mathop{\bigvee}\limits_s S^{m-1}_s \mathop{\cup}\limits_{\widehat{\phi}} S^m 
\]
where the number of $(m-1)$-spheres~$S^{m-1}_s$ is equal to the number of $(m-1)$-cells of~$M$ and
\[
\widehat{\phi} : S^{m-1} \mathop{\longrightarrow}\limits^{\phi} M(m-1) \longrightarrow M(m-1)/M(m-2)
\simeq \mathop{\bigvee}\limits_s S^{m-1}_s
\]
is the projection of the attaching map~$\phi$.

To derive the proposition, we need to show that $\phi$ is null-homotopic. 
Consider the triple $(M, M(m-1), M(m-2))$ and the corresponding long exact sequence with $\Z$-coefficients
\[
\dots 0 \mathop{\longrightarrow}\limits^{i_*} H_m(M, M(m-2)) \mathop{\longrightarrow}\limits^{j_*} 
H_m(M, M(m-1)) \mathop{\longrightarrow}\limits^{\partial} H_{m-1}(M(m-1), M(m-2)) \longrightarrow \dots  .
\]
From the long exact sequence of the pair $(M, M(m-2))$, we obtain that 
\[
H_m(M, M(m-2)) \simeq  H_m(M).
\]
The orientability of $M$ implies that 
\[
H_m(M, M(m-1)) \simeq  H_m(M).
\]
Thus, $j_*$ is an isomorphism, which implies that $\partial = 0$. 

Thinking of the homology groups $H_{k}(M(k),M(k-1);\Z)$ as free abelian groups with basis the $k$-cells~$e_\alpha^k$ of~$M$, the cellular boundary formula, see~\cite[\S2.2]{hatcher}, gives
\[
\partial(e^m) = \sum_s {\rm deg}(\phi_s) \, e_s^{m-1}
\]
where $\phi_s: S^{m-1} \to M(m-1) \to S_s^{m-1}$ is the composite of the attaching map~$\phi$ of the $m$-cell~$e^m=D^m$ with the quotient map collapsing $M(m-1) \setminus e_s^{m-1}$ to a point.
(Note that $\phi_s$ factorizes through~$\hat{\phi}$.)
Since $\partial =0$, every map~$\phi_s$ is contractible.
Hence $\phi$ is null-homotopic as desired.
\forget
The last means that the homology homomorphisme
$$
\widehat{\phi}_* : H_{m-1}(S^{m-1}) \longrightarrow H_{m-1}( M(m-1)/M(m-2)) 
\simeq H_{m-1}( M(m-1), M(m-2)) 
$$
iduces by $\widehat{\phi}$ is trivial.

Hurewicz homomorphisme
$$
h_* : \pi_{m-1}\Big{(}\mathop{\bigvee}\limits_s S^{m-1}_s\Big{)} \longrightarrow 
H_{m-1}\Big{(}\mathop{\bigvee}\limits_s S^{m-1}_s\Big{)}
$$ 
is an isomorphisme for $m \geq 3 $ so $\widehat{\phi}$ is trivial in homotopy if it is trivial in homology. This completes the proof.
\forgotten
\end{proof}

We can now prove Theorem~\ref{theo:stabilisation}. 

\begin{proof}[Proof of Theorem~\ref{theo:stabilisation}.]
Choose a cell decomposition of~$M_2$ with only one cell of maximal dimension~$m$, which is coherent with the triangulation of the pseudomanifold.
Denote by~\mbox{$\Cone(M_2(m-2))$} the cone over the $(m-2)$-skeleton~$M_2(m-2)$ of~$M_2$.
By Proposition~\ref{prop:bouquet}, the space 
\[
M_2 \mathop{\cup}\limits_{M_2(m-2)} \Cone(M_2(m-2))
\]
obtained by gluing the cone~$\Cone(M_2(m-2))$ to~$M_2$ along~\mbox{$M_2(m-2)$} is homotopy equivalent to a finite bouquet of spheres
$$
M_2 \mathop{\cup}\limits_{M_2(m-2)} \Cone(M_2(m-2)) \simeq \mathop{\bigvee}\limits_s S^{m-1}_s \bigvee S^m
$$
with only one $m$-dimensional sphere~$S^m$.
Thus, the natural inclusion $M_1 \setminus B^m \subseteq M_1 \sharp M_2$ extends to the missing ball~$B^m$ and gives rise to an $m$-monotone map, see Definition~\ref{def:monotone},
\begin{equation}\label{eq:f}
f: M_1 \to  M_1 \sharp M_2 \mathop{\cup}\limits_{M_2(m-2)}  \Cone(M_2(m-2)).
\end{equation}
Fix~$\e > 0$.
Consider a metric~$g$ on~$M_1 \sharp M_2$ with $\vol(M_1 \sharp M_2, g)=1$, which is $\e$-extremal, that is, 
\begin{equation} \label{eq:H1*H2}
\ent_H(M_1 \sharp M_2, g)^m \leq  
\Omega_H(M_1 \sharp M_2) + \e. 
\end{equation}
Extend the metric~$g$ to a metric~$g'$ on $M_1 \sharp M_2 \mathop{\cup}\limits_{M_2(m-2)}  \Cone(M_2(m-2))$ as follows.
First, observe that
$$
\Cone(M_2(m-2)) = M_2(m-2) \times [0, 1]\slash M_2(m-2) \times \{1\}.
$$ 
The extension~$g'$ of~$g$, which agrees with~$g$ on $M_1 \sharp M_2$, is defined on $\Cone(M_2(m-2))$ by
\[
g' = 
\begin{cases}
g_{|M_2(m-2) } + 10D \, dt^2  & \text{if } 0 \leq t \leq \frac{1}{2} \\
4(1-t)^2 g_{|M_2(m-2) } + 10D \, dt^2  & \text{if } \frac{1}{2} \leq t \leq 1
\end{cases}
\]
where $D = \diam \left( M_2(m-2), g_{|M_2(m-2)} \right)$.
Technically, the metric~$g'$ is singular, but it still induces a distance~$\dist_{g'}$.
Note also that by dimensional reasons 
$$
\vol(M_1 \sharp M_2 \mathop{\cup}\limits_{M_2(m-2)}  \Cone(M_2(m-2)), g') = \vol(M_1 \sharp M_2, g) = 1.
$$
By construction of~$g'$, the natural inclusion 
\begin{equation}\label{eq:inclusion.isom.dist}
i: M_1 \sharp M_2 \hookrightarrow 
M_1 \sharp M_2 \mathop{\cup}\limits_{M_2(m-2)}  \Cone(M_2(m-2))
\end{equation}
is distance-preserving.
That is, for every $p_1, p_2 \in M_1 \sharp M_2$, we have
\[
\dist_g(p_1, p_2) = \dist_{g'}(i(p_1), i(p_2)).
\]

Since $m \geq 3$, the composite map
\begin{equation}\label{eq:inclusion.isom}
M_1 \setminus B^m \subseteq  M_1 \sharp M_2 \hookrightarrow 
M_1 \sharp M_2 \mathop{\cup}\limits_{M_2(m-2)}  \Cone(M_2(m-2))
\end{equation}
induces an isomorphism between the fundamental groups.
Denote by $G_1 < G = \pi_1(M_1 \sharp M_2)$ the image of~$\pi_1(M_1)$ in~$\pi_1(M_1 \sharp M_2)$. 

\medskip

Let $q_1 \in M_1 \setminus B^m \subseteq  M_1 \sharp M_2$.
Since $\Cone(M_2(m-2))$ is simply connected and the map~\eqref{eq:inclusion.isom} is distance-preserving, every loop $\gamma \subseteq M_1 \sharp M_2 \mathop{\cup}\limits_{M_2(m-2)}  \Cone(M_2(m-2))$ based at~$q_1$ is homotopic to a loop $\gamma' \subseteq M_1 \sharp M_2$ based at the same point such that 
\begin{equation}\label{eq:length.domination}
\length_g(\g') \leq \length_{g'}(\g).
\end{equation}

\medskip

The group $G /H$ acts on the cover $\widehat{M}$ of~$M_1 \sharp M_2$ with fundamental group~$H$.
Similarly, the group $G_1 / H_1$ acts on the cover $\widehat{M}'$ of~$M_1 \sharp M_2 \mathop{\cup}\limits_{M_2(m-2)}  \Cone(M_2(m-2))$ with fundamental group~$H_1$.
Let $\widehat{g}$ and~$\widehat{g}'$ be the metrics on $\widehat{M}$ and~$\widehat{M}'$ induced by~$g$ and~$g'$.
Fix some lifts~$q \in \widehat{M}$ and $q' \in \widehat{M}'$ of~$q_1$.
Denote by $B_H(t, q; \widehat{g})$ and $B_{H_1}(t, q'; \widehat{g}')$ the balls of~$\widehat{M}$ and~$\widehat{M}'$ of radius~$t$ centered at~$q$ and~$q'$.
Since $G_1/H_1 < G/H$, see~\eqref{eq:group.inclusion}, it follows from~\eqref{eq:length.domination} that for every $t\geq 0$
\begin{equation}\label{eq:boule.comparison}
\vert B_{H_1}(t, q'; \widehat{g}') \cap q' \cdot (G_1/H_1) \vert \leq 
\vert B_H(t, q; \widehat{g}) \cap q \cdot (G/H) \vert.
\end{equation}

Applying Lemma~\ref{lem:nonexpanding} to the $m$-monotone map~$f$, see~\eqref{eq:f}, we derive a polyhedral metric~$g_\e$ on~$M_1$ such that the map~$f$ is nonexpanding and
\begin{equation}\label{eq:vol.comparison}
\vol (M_1, g_{\e}) < \vol(M_1 \sharp M_2,g) + \e.
\end{equation}

The group~$G_1/H_1$ acts both on~$\widehat{M}'$ and on the cover~$\widehat{M}''$ of~$M_1$ with fundamental group~$H_1$.
Let $\widehat{g}''_\e$ be the metric on~$\widehat{M}''$ induced by~$g_\e$.
Fix a lift~$q''$ of~$q_1$ in~$\widehat{M}''$.
The nonexpanding map~$f$ lifts to a $(G_1/H_1)$-equivariant, nonexpanding map $\widehat{f}:\widehat{M}'' \to \widehat{M}'$.
This implies that the ball $B_{H_1}(t,q'';\widehat{g}''_\e)$ of~$\widehat{M}''$ satisfies
\begin{equation}\label{eq:boule.comparison2}
\vert B_{H_1}(t, q''; \widehat{g}_{\e}'') \cap q '' \cdot (G_1/H_1) \vert \leq 
\vert B_{H_1}(t, q'; \widehat{g}') \cap q' \cdot (G_1/H_1)\vert . 
\end{equation}

Combining the bounds~\eqref{eq:boule.comparison} and~\eqref{eq:boule.comparison2}, we derive the following inequalities on the exponential growth rates of the orbits of $G_1/H_1$ and $G/H$ 
\[
\ent_{H_1}(M_1,g_\e) \leq \ent_{H_1}(M_1 \sharp M_2 \mathop{\cup}\limits_{M_2(m-2)}  \Cone(M_2(m-2)),g') \leq \ent_H(M_1 \sharp M_2,g)
\]
Since $\vol(M_1 \sharp M_2,g)=1$ and $\vol (M_1, g_{\e}) < 1 + \e$, see~\eqref{eq:vol.comparison}, this estimate combined with~\eqref{eq:H1*H2} yields the desired bound
\[
\Omega_{H_1} (M_1) \leq \Omega_H (M_1 \sharp M_2).
\]
\end{proof}
\begin{remark}
The proof of Theorem~\ref{theo:stabilisation} does not apply when $M_2$ is nonorientable.
The conclusion is unclear in this case.
\end{remark}

\begin{remark}
Theorem~\ref{theo:stabilisation} and Corollary~\ref{cor:stabilisation1} hold true if one replaces the relative minimal entropy to the power~$m$, namely~$\Omega_H(M)$, with the relative systolic volume~$\sigma_H(M)$, defined in a similar way as in~\eqref{eq:sigmaf} (where $\sys_H(M,g)$ is the length of the shortest loop of~$M$ whose homotopy class does not lie in~$H$).
The proofs are similar, except that the bound on the volume entropy~$\ent_{H_1}(M_1,g_\varepsilon)$ at the end of the proof of Theorem~\ref{theo:stabilisation} should be replaced with
\[
\sys_{H_1}(M_1,g_\e) \geq \sys_{H_1}(M_1 \sharp M_2 \mathop{\cup}\limits_{M_2(m-2)}  \Cone(M_2(m-2)),g') \geq \sys_H(M_1 \sharp M_2,g).
\]
\end{remark}

\subsection{Fundamental class of finite order}

\mbox{ } 

\medskip

The following result is a direct application of Theorem~\ref{theo:stabilisation}.
\begin{proposition}\label{prop:stabilisation2}
Let $M$ be an orientable connected closed manifold of dimension~$m \geq 3$ and $f : M \to K(\pi_1(M), 1)$ be its characteristic map.
Suppose $f_{*}([M]) \in H_m(\pi_1(M);\Z)$ is a finite order homology class.
Then
$$
\Omega(M) = 0.
$$ 
\end{proposition}
\begin{proof}
Let $N = M\sharp \dots \sharp M$ be the manifold obtained by taking the connected sum of~$k$ copies of~$M$.
Consider the map $F= f \vee \dots \vee f : N \to K(\pi_1(M), 1)$ obtained by collapsing each attaching sphere to a point and by applying~$f$ to each factor~$M$ in the bouquet~$M\vee \dots \vee M$.
The class $F_{*}([N])$ is equal to~$k \, {\bf a}$, where ${\bf a} = f_{*}([M])$.
Suppose $k \, {\bf a} = 0$.
Since $N$ is an orientable connected closed manifold and the homomorphism $F_{*}:\pi_1(N) \to \pi_1(M)$ induced by~$F$ is surjective, we derive from~\cite[Theorem~10.2]{Brunnbauer08}, see~\eqref{eq:thm10.2}, that
$$
\Omega_{\ker F_{*}} (N) = 0.
$$ 
Apply Theorem~\ref{theo:stabilisation} to $M_1 = M$ and the connected sum $M_2 = M\sharp \dots \sharp M$ of $k-1$ copies of~$M$ by taking $H = \ker F_{*}$.
This immediately leads to the desired result. 
\end{proof}
\begin{example}
Every manifold~$M$ with fundamental group $SL(2, \Z)= \Z_4 *_{\Z_2} \Z_6$ or $PSL(2, \Z) \simeq \Z_2 * \Z_3$ has zero minimal volume entropy, that is, $\Omega (M) = 0$. 
Indeed, the homology of an amalgamated product can be computed through a Mayer-Vietoris sequence involving the homology groups of its factors.
Since the homology of every cyclic group is composed of finite groups (except in dimension zero), the same holds for the homology groups of $SL(2, \Z)$ and $PSL(2, \Z)$, see~\cite[Theorem 4.1.1]{knudson}.
\end{example}

\subsection{Volume entropy semi-norm comparison}

\mbox{ }

\medskip

Let us give an application of Corollary~\ref{cor:stabilisation1}.

\begin{proposition}\label{prop:indice.fini}
Let $G$ be a finitely presented group and $H$ be a finite index subgroup of~$G$.
Let $m \geq 3$.
Suppose that the natural inclusion $i: H \hookrightarrow G$ induces an isomorphism between the $m$-dimensional rational homology groups
$$
(i_*)_m:  H_m(H; \Q) \to  H_m(G; \Q).  
$$
Then for every homology class ${\bf a} \in H_m(H; \Z)$ 
$$
\Vert {\bf a} \Vert_E =  \Vert i_*{(\bf a}) \Vert_E.
$$
\end{proposition}
\begin{proof}
Still denote by $i:K(H,1) \to K(G,1)$ the characteristic map induced by the natural inclusion $i: H \hookrightarrow G$.
By~\eqref{eq:Omegaf}, for every integer $k \geq 1$, we have
\[
\Omega(k \, i_*({\bf a})) = \Omega(i_*(k \, {\bf a})) \leq \Omega(k \, {\bf a}).
\]
Hence
\[
\Vert {\bf a} \Vert_E \geq  \Vert i_*({\bf a}) \Vert_E.
\]
Thus, we only have to show the converse inequality.

\medskip

Denote by $p$ the number of generators in~$G$ and by $d=[G:H]$ the index of~$H$ on~$G$.
Let
$$
M_2 = (S^1\times S^{m-1}) \sharp  (S^1\times S^{m-1})  \sharp  \dots  \sharp  (S^1\times S^{m-1})
$$
be the connected sum of $p$ copies of~$S^1\times S^{m-1}$.
Let $f_2 : M_2 \to K(G, 1)$ be a map inducing a surjective homomorphism between the fundamental groups.
Observe that $\Omega(M_2) =0$ and $(f_2)_*([M_2])=0 \in H_m(G;\Z)$.

\medskip

Fix $\e > 0$.
By definition of~$\Vert \cdot \Vert_E$, for every integer $k \geq 1$, there exists a map
$
f_1:M_1 \to K(G,1)
$
defined on an orientable connected closed $m$-pseudomanifold representing the class $k \, i_*({\bf a}) \in H_m(G; \Z)$ such that
\[
\Omega_{\ker (f_1)_*} (M) \leq k \, (\Vert i_*({\bf a}) \Vert_E + \e).
\]
Consider the connected sum $M = M_1 \sharp M_2$ and the natural map $f   = f_1 \vee f_2 : M \to K(G, 1)$ obtained from~$f_1$ and~$f_2$ by collapsing the attaching sphere into a point.
Note that 
\[
\ker f_* = \langle \ker (f_1)_* * \ker (f_2)_* \rangle.
\]
Since $(f_2)_*([M_2]) = 0 \in H_m(G;\Z)$, the map $f:M \to K(G,1)$ still represents the class~$k \, i_*({\bf a}) \in H_m(G;\Z)$.
Since $\Omega_{\ker (f_2)_*}(M_2) =0$, we deduce from Corollary~\ref{cor:stabilisation1} that
\begin{equation} \label{eq:ext-a}
\Omega_{\ker f_*} (M) \leq k \, (\Vert i_*({\bf a}) \Vert_E + \e).
\end{equation}

Since $f_2$ induces a surjective homomorphism between the fundamental groups, the same holds for the map~$f$.
Let $\widehat{M}$ be the cover of~$M$ of fundamental group~$f_{*}^{-1}(H)$.
Denote by 
$$
\widehat{f} : \widehat{M} \to K(H, 1)
$$
the corresponding lift of~$f$.
Let ${\bf b} = \widehat{f}_*([\widehat{M}]) \in H_m(H;\Z)$.
Since $H$ is of index~$d$ in~$G$, the cover $\pi:\widehat{M} \to M$ is of degree~$d$.
Thus, $\pi_*([\widehat{M}]) = d \, [M]$.
Still denote by $i:K(H,1) \to K(G,1)$ the characteristic map induced by the natural inclusion $i:H \hookrightarrow G$.
It follows from the commutation relation $i \circ \widehat{f} = f \circ \pi$ that
\[
i_*({\bf b}) = i_*(\widehat{f}_*([\widehat{M}]) = d \, f_*([M]) = d  k \, i_*({\bf a}).
\]
Since the natural inclusion $i:H \hookrightarrow G$ induces an isomorphism between the $m$-dimensional rational homology groups, we deduce that ${\bf b} = d k \,${\bf a} + ${\bf c}$ where ${\bf c} \in \text{Tor} H_m(H;\Z)$.
Thus, 
\[
\Vert {\bf b} \Vert_E = d k \, \Vert {\bf a} \Vert_E.
\]

Let $g$ be an $\e$-extremal metric on~$M$, that is,
\begin{equation} \label{eq:ext-2}
\Omega_{\ker f_*} (M, g) \leq \Omega_{\ker f_*} (M) + \e.
\end{equation}
Denote by~$\widehat{g}$ the lift of~$g$ on~$\widehat{M}$.
The cover of~$(M,g)$ of fundamental group~$\ker f_{*}$ is isometric to the cover of~$(\widehat{M}, \widehat{g})$ of fundamental group~$\ker \widehat{f}_{*}$.
Thus, the exponential growth rates of the volume of balls in the two coverings are equal.
Since $\pi:\widehat{M} \to M$ is of degree~$d$, we have $\vol(\widehat{M}) = d \, \vol(M)$.
Therefore, 
$$
\Omega_{\ker \widehat{f}_{*}}(\widehat{M}, \widehat{g}) = d \, \Omega_{\ker f_{*}} (M, g).
$$
Now, by construction, $(\widehat{M},\widehat{f})$ represents~${\bf b}$.
Hence,
\[
d k \, \Vert {\bf a} \Vert_E = \Vert {\bf b} \Vert_E \leq \Omega_{\ker \widehat{f}_{*}}(\widehat{M}, \widehat{g}) = d \, \Omega_{\ker f_*}(M,g).
\]
This inequality combined with the bounds~\eqref{eq:ext-a} and~\eqref{eq:ext-2} yields
\[
\Vert {\bf a} \Vert_E \leq \Vert i_*({\bf a}) \Vert_E + 2\e.
\]
Hence the desired result by letting $\e$ go to zero.
\end{proof}

\section{Functorial properties of the volume entropy semi-norm} \label{sec:functorial}

In this section, we present functorial properties of the volume entropy semi-norm and observe similarities with the ones satisfied by the simplicial volume.

\begin{theorem}\label{theo:functorial}
\mbox{ }
\begin{enumerate}
\item Let $f:X \to Y$ be a continuous map between two path-connected topological spaces.
Then for every ${\bf a} \in H_m(X;\R)$
\[
\Vert f_*({\bf a}) \Vert_E \leq \Vert {\bf a} \Vert_E.
\]
\label{f1}
\item Let $f: M \to K(\pi_1(M),1)$ be the classifying map of an orientable connected closed manifold~$M$.
Then
\[
\Vert f_*([M]) \Vert_E = \Vert M \Vert_E.
\]
\label{f2}
\item Let $f: M \to N$ be a degree~$d$ map between two orientable connected closed manifolds.
Then
\[
\Vert M \Vert_E \geq \vert d \vert \, \Vert N \Vert_E.
\]
\label{f3}
\item Let $f:M \to N$ be a $d$-sheeted covering map between two oriented connected closed manifolds.
Then 
\[
\Vert M \Vert_E = d \, \Vert N \Vert_E.
\]
\label{f4}
\item Let $M_1$ and $M_2$ be two orientable connected closed manifolds of dimension~$m \geq 3$.
Then
\begin{equation} \label{eq:subadd}
\Vert M_1 \sharp M_2 \Vert_E \leq \Vert M_1 \Vert_E + \Vert M_2 \Vert_E.
\end{equation}
\label{f5}
\item Let $M$ be an orientable connected closed $m$-manifold with a negatively curved locally symmetric metric~$g_0$.
Then
\[
\Vert M \Vert_E = \Omega(M,g_0).
\]
In particular, if $M$ is a closed genus~$g$ surface then
\[
\Vert M \Vert_E = \Vert M \Vert_\Delta = 4 \pi (g-1).
\]
\label{f6}
\end{enumerate}
\end{theorem}

\begin{remark}
The properties \eqref{f1}-\eqref{f5} are also satisfied by the simplicial volume.
However, the simplicial volume is additive under connected sum, see~\cite{gro82}.
That is, there is equality in~\eqref{eq:subadd} if one replaces the volume entropy semi-norm with the simplicial volume.
This leads to the following questions.
Is there equality in~\eqref{eq:subadd}?
Similarly, is there equality in~\eqref{eq:semiadditivite1}?
\end{remark}

\begin{remark}
It follows from~\eqref{f3} that both the simplicial volume and the volume entropy semi-norm of an orientable connected closed manifold admitting a map to itself of degree different from $0$ and~$\pm 1$ are equal to zero.
In a different direction, by Theorem~\ref{theo:E>S}, both the simplicial volume and the volume entropy semi-norm do not vanish for orientable connected closed manifolds admitting a negatively curved Riemannian metric, see~\cite{gro82}.
\end{remark}

\begin{proof}[Proof of Theorem~\ref{theo:functorial}] \mbox{ }

\eqref{f1} Observe that if $(M,\varphi)$ is a geometric cycle representing~${\bf a} \in H_m(X;\R)$ then $(M,f \circ \varphi)$ is a geometric cycle representing~$f_*({\bf a}) \in H_m(Y;\R)$.
Moreover, $\omega_{\ker f_*}(M) \geq \omega_{\ker (\varphi \circ f)_*}(M)$ since $\ker f_* \subseteq \ker (\varphi \circ f)_*$.
This immediately implies~\eqref{f1}.

\medskip

\eqref{f2} For $m=2$, the assertion is obvious since the classifying map is the identity map.

Suppose that $m \geq 3$.
The inequality $\Vert f_*([M]) \Vert_E \leq \Vert M \Vert_E$ follows from~\eqref{f1}.
For the reverse inequality, consider the manifold~$M_k = M \sharp \cdots \sharp M$ defined as the connected sum of $k$ copies of~$M$ and the degree~$k$ map $f_k :M_k \to M$ contracting all attaching spheres into a point.
Clearly, 
\begin{equation} \label{eq:f_k}
\Omega_{\ker {(f_k)}_*}(M_k) \geq \Omega(k \, [M]).
\end{equation} 
The composite map 
\[
F_k=f \circ f_k : M_k \to K(\pi_1(M),1)
\]
represents the class $k \, f_*([M]) \in H_m(\pi_1(M);\Z)$, that is, ${(F_k)}_*([M_k]) = k \, f_*([M])$.
Observe also that it is $\pi_1$-surjective.
By \cite[Theorem 10.2]{Brunnbauer08}, this implies that
\begin{equation}\label{eq:Omega(k[M])}
\Omega(k \, f_*([M])) = \Omega_{\ker {(F_k)}_*}(M_k).
\end{equation} 
Since $f$ is $\pi_1$-injective, $\Omega_{\ker {(F_k)}_*}(M_k) = \Omega_{\ker {(f_k)}_*}(M_k)$.
Combined with~\eqref{eq:f_k} and~\eqref{eq:Omega(k[M])}, we derive 
\[
\frac{\Omega(k \, f_*([M]))}{k} \geq \frac{\Omega(k \, [M])}{k}.
\]
By letting $k$ go to infinity, we obtain $\Vert f_*([M]) \Vert_E \geq \Vert M \Vert_E$, which implies~\eqref{f2}.

\medskip

\eqref{f3} By definition, the assertion~\eqref{f3} immediately follows from~\eqref{f1}.

\medskip

\eqref{f4} 
Let $(Q,\psi)$ be a geometric cycle representing the class $k \, [N] \in H_m(N;\Z)$.
By adding handles to~$Q$ (with the meridian circle of each handle collapsed to a point when $m=2$, which amounts to taking connected sums with spheres where two antipodal points are identified) and mapping them to a generating set of~$\pi_1(N)$ if necessary, we can assume that the map $\psi:Q \to N$ is $\pi_1$-surjective.
By Theorem~\ref{th:sou.additif} (and Remark~\ref{rem:sou.additif} when $m = 2$), adding such handles does not increase the (relative) minimal volume entropy of the geometric cycle.
Now, denote by $P \to Q$ the covering of~$Q$ corresponding to the subgroup~$(\psi_*)^{-1}({\rm Im}\, f_*)$.
Note that $P$ is an orientable connected closed $m$-pseudomanifold and that the covering map $P \to Q$ is of degree~$d$.
The map $\psi:Q \to N$ lifts to a map $\varphi:P \to M$ such that the following diagram 
\[ \begin{tikzcd}
P \arrow{r}{\varphi} \arrow{d} & M \arrow{d}{f} \\
Q \arrow{r}{\psi}& N
\end{tikzcd}
\]
commutes.
Observe that the geometric cycle~$(P,\varphi)$ represents the class $k \, [M] \in H_m(M;\Z)$.

Fix a piecewise Riemannian metric on~$Q$ and lift it to~$P$.
Since $P \to Q$ is a $d$-sheeted covering map and $\ker \varphi_* = \ker \psi_*$, we have $\Omega_{\ker \varphi_*}(P) \leq d \, \Omega_{\ker \psi_*}(Q)$.
Thus,
\[
\Omega(k \, [M]) \leq d \, \Omega(k \, [N]).
\]
Dividing this inequality by~$k$ and letting $k$ go to infinity, we obtain 
\[
\Vert M \Vert_E \leq d \, \Vert N \Vert_E.
\]
The reverse inequality $\Vert M \Vert_E \geq d \, \Vert N \Vert_E$ follows from~\eqref{f3}.

\medskip

\eqref{f5} Let $K_i=K(\pi_1(M_i),1)$ be a classifying space for~$M_i$.
Since $m \geq 3$, the bouquet $K = K_1 \vee K_2$ is a classifying space for~$M_1 \sharp M_2$.
By the Mayer-Vietoris theorem, we have
\[
H_m(K;\Z) = H_m(K_1;\Z) \oplus H_m(K_2;\Z).
\]
Denote by $[M_i]_K \in H_m(K;\Z)$ the image of the fundamental class of~$M_i$ under the homology homomorphism induced by the map $f_i^K:M_i \to K$, which is defined as the composite of the classifying map $f_i:M_i \to K_i$ and the inclusion map $K_i \hookrightarrow K = K_1 \vee K_2$.
Observe that 
\[
f_*([M_1 \sharp M_2]) = [M_1]_K + [M_2]_K
\]
where $f:M_1 \sharp M_2 \to K$ is the classifying map of~$M_1 \sharp M_2$.
By~\eqref{f2} and the triangle inequality of the volume entropy semi-norm, see Corolllary~\ref{coro:seminorm}, we derive
\[
\Vert M_1 \sharp M_2 \Vert_E = \Vert f_*([M_1 \sharp M_2]) \Vert_E \leq \Vert [M_1]_K \Vert_E + \Vert [M_2]_K \Vert_E.
\]
By~\eqref{f1}, we also have $\Vert [M_i]_K \Vert_E \leq \Vert [M_i] \Vert_E = \Vert M_i \Vert_E$.
Hence the result.

\medskip

\eqref{f6} 
The proof proceeds from a mild improvement on the minimal volume entropy estimate for closed manifolds admitting nonzero maps onto closed negatively curved locally symmetric manifolds, see~\cite{BCG95}.
This mild improvement, leading to~\eqref{eq:OOO}, was carried out in~\cite[Theorem~2.5]{Sambusetti99} for $n \geq 3$ with the construction of a nonexpanding-volume map following~\cite{BCG95}.
Our approach is similar, except that it rests on the calibration argument of~\cite{BCG95} (which can be applied to pseudomanifolds) and applies to both cases $n \geq 3$ and $n=2$.
We refer to~\cite{BCG95} for the notations, the definitions and further details.

Let $f=Y \to X=M$ be a map from an orientable connected closed $m$-pseudomanifold representing $k \, [X]$.
The map $f$ lifts to a map $\overline{f}:\overline{Y} \to \tilde{X}$, where $\overline{Y}$ is the covering of~$Y$ with $\pi_1(\overline{Y}) = \ker f_*$.
(This is the main difference with~\cite[\S8]{BCG95}, where the map is lifted to $\tilde{Y} \to \tilde{X}$.)
Given a piecewise Riemannian metric~$g$ on~$Y$, denote by~$\overline{g}$ the lifted metric on~$\overline{Y}$.
Fix~$c>0$.
Consider the $\pi_1(Y)/\ker f_*$-equivariant map $\Psi_c:\overline{Y} \to S_+^\infty \subset L^2(\partial \tilde{X},d\theta)$ defined as
\[
\Psi_c(y,\theta) = \left( \frac{\int_{\overline{Y}} e^{-c \, d_{\overline{g}}(y,z)} p_0(\overline{f}(z),\theta) \, dv_{\overline{g}}(z)}{\int_{\overline{Y}} e^{-c \, d_{\overline{g}}(y,z)} \, dv_{\overline{g}}(z)} \right)^\frac{1}{2}
\]
where $p_0$ is the Poisson kernel of~$(\tilde{X},\tilde{g}_0)$, see~\cite{BCG95}.
The arguments of~\cite{BCG95} show that the map~$\Psi_c$ is only defined when $c>\ent_{\ker f_*}(Y,g)$ and that
\begin{equation} \label{eq:det}
\sqrt{{\det}_{\overline{g}}(g_{\Psi_c})} \leq \frac{c^n}{(4n)^{n/2}}
\end{equation}
where $g_{\Psi_c}$ is the pull-back under~$\Psi_c$ of the Hilbert metric on~$L^2(\partial \tilde{X},d\theta)$.
Denote by~$\pi:S_+^\infty \to \tilde{X}$ the barycenter map, \ie, $\pi(\rho)={\rm bar}(\rho^2(\theta) \, d\theta)$, see~\cite{BCG95}.
The following calibration result was established in~\cite[Proposition~5.7]{BCG95} for $n \geq 3$ and in~\cite[Theorem~6.2]{BCG95} for $n=2$.
Let $\omega_0$ be the volume form on~$(\tilde{X},\tilde{g}_0)$.
The $\pi_1(X)$-equivariant closed $m$-form~$\pi^* \omega_0$ on~$S_+^\infty$ calibrates the embedding $\Phi_0:\tilde{X} \to S_+^\infty \subset L^2(\partial \tilde{X},d\theta)$ defined as $\Phi_0(x) = \sqrt{p_0(x,\cdot)}$.
Furthermore, 
\begin{equation} \label{eq:comass}
{\rm comass}(\pi^* \omega_0) = \frac{(4n)^{n/2}}{\ent(X,g_0)^n}.
\end{equation}
Now, the map~$\pi \circ \Psi_c$ is homotopic to~$\pi \circ \Psi_0$, where $\Psi_0 = \Phi_0 \circ \overline{f}$, through equivariant maps from~$\overline{Y}$ to~$\tilde{X}$.
By~\eqref{eq:det} and~\eqref{eq:comass}, we derive 
\[
|(\pi \circ \Psi_c)^* (\omega_0) | = |\Psi_c^*(\pi^* \omega_0)| \leq \left( \frac{c}{\ent(X,g_0)} \right)^n \, |\omega_{\overline{g}}|.
\]
By a calibration argument, we deduce that
\[
|{\rm deg}(f)| \cdot \vol(X,g_0) \leq \int_Y \left| (\pi \circ \Psi_c)^* (\omega_0) \right| \leq \left( \frac{c}{\ent(X,g_0)} \right)^n \, \vol(Y,g)
\]
passing the volume form to the quotient.
As $c$ go to~$\ent_{\ker f_*}(Y,g)$, we obtain
\begin{equation} \label{eq:OOO}
\Omega(X,g_0) \leq \frac{\Omega_{\ker f_*}(Y,g)}{k}.
\end{equation}
Taking the infimum over all piecewise Riemannian metrics~$g$ on~$Y$ and over all geometric cycles~$(Y,f)$ representing~$k \, [X]$, we derive $\Omega(X,g_0) \leq \Vert X \Vert_E$ after letting $k$ go to infinity.

The reverse inequality is obvious.
\end{proof}

\begin{remark}
The assertion~\eqref{f6} can be extended to the following case.
If $M$ is an orientable connected closed $2m$-manifolds given by a compact quotient of the product of~$m$ hyperbolic planes then
\[
\Vert M \Vert_E = \Omega(M,g_0)
\]
where $g_0$ is the unique locally symmetric metric of minimal volume entropy on~$M$ among all locally symmetric metrics of given volume.
Indeed, the proof can be adapted to follow the argument of~\cite{merlin} based on the same calibration method as~\cite{BCG95}.
\end{remark}

We need a couple of definitions to present the next result. 

\begin{definition}
For $i=1,2$, let $V_i$ be a $\R$-vector space endowed with a semi-norm~$\Vert . \Vert_i$. 
The tensor product $V_1\otimes V_2$ inherits the semi-norm~$\Vert . \Vert_{\otimes}$ given by the tensor product of~$\Vert . \Vert_1$ and~$\Vert . \Vert_2$, see \cite{Schaefer66}.
By definition, for every~$u \in V_1\otimes V_2$, we have
\begin{equation}\label{eq:tensor}
\Vert u \Vert_{\otimes} = \mathop{\inf} \left\{ \mathop{\sum}_s \Vert x_s \Vert_1 \, \Vert y_s \Vert_2 \mid
u =  \mathop{\sum}_s x_s \otimes y_s \right\}
\end{equation}
where the infimum is taken over all the representations of~$u$ by finite sums of simple tensor products.

In the sequel, we endow the direct sum of semi-normed vector spaces with the direct sum of the semi-norms.
With this convention, the graded vector space of the real homology~$H_*(X; \R)$ of a path-connected topological space~$X$ is endowed with the \emph{graded volume entropy semi-norm}~$\Vert . \Vert_{*E}$ given on each homogenous component by
$$
\Vert {\bf a} {\Vert_m}_E = \frac{1}{m^m} \, \Vert {\bf a} \Vert_E
$$
for every ${\bf a} \in H_m(X; \R)$.

The real homology of the direct product~$X_1 \times X_2$ of two path-connected topological spaces~$X_1$ and~$X_2$ is naturally endowed with two semi-norms.
The first one is the usual volume entropy semi-norm~$\Vert . \Vert_{*E}$. 
The second one is defined via K\"unneth's formula 
$$
H_m(X_1 \times X_2; \R) \simeq \mathop{\oplus}\limits_{i+j = m} H_i(X_1; \R) \otimes H_j(X_2; \R)
$$
as the tensor product norm, see~\eqref{eq:tensor}, of the graded volume entropy semi-norms~$\Vert . \Vert_{*E}^{(i)}$ on~$H_*(X_i; \R)$. 
It is denoted by~$\Vert . \Vert^{\otimes}_E$. 
\end{definition}

We can now state our next result.

\begin{theorem}
Let $X_1$ and~$X_2$ be two path-connected topological spaces. 
Then, for every ${\bf a} \in H_*(X_1 \times X_2; \R)$, the following inequality holds
\begin{equation*}\label{eq:tensor.1}
\Vert {\bf a} \Vert_{*E} \leq \Vert {\bf a} \Vert^{\otimes}_E.
\end{equation*}
\end{theorem}
\begin{proof}
It is enough to prove the inequality for homogeneous elements.
Every homology class ${\bf a} \in  H_m(X_1 \times X_2; \R)$ admits a representation as a sum of simple tensor products
\begin{equation}\label{eq:tensor.2}
{\bf a} = \mathop{\sum}_s {\bf x}_s \otimes {\bf y}_s
\end{equation}
where ${\bf x}_s \in H_{i_s}(X_1; \R)$ and ${\bf y}_s \in H_{j_s}(X_2; \R)$ with $i_s + j_s = m$.
By the triangle inequality,
\begin{equation}\label{eq:tensor.3}
\Vert {\bf a} \Vert_{*E} \leq  \mathop{\sum}_s \Vert {\bf x}_s \otimes {\bf y}_s \Vert_{*E}.
\end{equation}
By multiplying if necessary the homology class~${\bf a}$ by an appropriated natural number, we can suppose that all classes ${\bf x}_s$ and ${\bf y}_s$ in~\eqref{eq:tensor.2} are represented by closed manifolds. 
Proposition~2.6 of~\cite{B93} implies that 
$$
\Vert {\bf x}_s \otimes {\bf y}_s \Vert_{*E} \leq \Vert {\bf x}_s \Vert_{*E} \, \Vert {\bf y}_s \Vert_{*E}.
$$
Plugging this bound in~\eqref{eq:tensor.3} and minimizing over all the simple tensor product representations~\eqref{eq:tensor.2}, we obtain
the desired inequality. 
\end{proof}

\section{Volume entropy semi-norm and simplicial volume} \label{sec:E<S}

In this section, we show that the volume entropy semi-norm of a homology class is bounded from above and below by its simplicial volume, up to some multiplicative constants depending only on the dimension of the homology class.
Therefore, the volume entropy semi-norm and the simplicial volume are equivalent homology semi-norms.

Throughout this section, let $X$ be a path-connected topological space.

\subsection{Geometrization of simplicial volume.}

\mbox{ }

\medskip

Let us introduce some topological invariants.

\begin{definition}
Let K be an $m$-dimensional topological space supplied with a finite pseudo-triangulation (also referred to as a \emph{pseudo-simplicial complex} or a \emph{$\Delta$-complex}, see~\cite[\S2.1]{hatcher}). 
Loosely speaking, the space~$K$ is a finite cell complex where the closure of each cell is homeomorphic to the standard simplex of the same dimension. 
In comparaison with usual simplicial complexes, a simplex in a pseudo-triangulation is not uniquely defined by its vertices.
An $m$-dimensional \emph{geometric $\Delta$-cycle} is a disjoint finite union of $m$-dimensional $\Delta$-complexes whose pseudo-triangulations satisfy the conditions \eqref{p1}, \eqref{p2} and~\eqref{p3} of Definition~\ref{def:pseudomanifold}.

The \emph{geometric complexity} of~$K$, denoted by~$\kappa(K)$, is the number of $m$-simplices of~$K$.
Define the \emph{geometric complexity} of a homology class ${\bf a} \in H_m(X;\Z)$ as
\[
\kappa({\bf a}) = \inf_P \kappa(P)
\]
where $P$ runs over the $m$-dimensional geometric $\Delta$-cycle representing~${\bf a}$.
That is, there is a map $h:P \to X$ such that $h_*([P]) = {\bf a}$, where the class~$[P]$ is the sum of the fundamental classes of the connected components of~$P$ with the appropriate orientations.
Define also the \emph{average geometric complexity} of~${\bf a}$ as
\begin{equation}\label{eq:complexite:moyenne1}
\kappa^{\infty}({\bf a}) = \mathop{\lim}_{n \to \infty}\frac{\kappa(n \, {\bf a})}{n}.
\end{equation}
Note that the function $\kappa(n \, {\bf a})$ is subadditive in~$n$, which ensures the existence of the limit~\eqref{eq:complexite:moyenne1}.
Furthermore, $\kappa^{\infty}({\bf a}) \leq \frac{\kappa(n \, {\bf a})}{n}$ for every $n \geq 1$.
\end{definition}

We will also need the following definition extension the notion of simplicial volume to homology classes with coefficients in more general rings.

\begin{definition}
Let $X$ be a topological space and $\AA = \Z \mbox{ or } \Q$.
For every ${\bf a} \in H_m(X;\AA)$, define
\[
\Vert {\bf a} \Vert_\Delta^\AA = \inf_c \Vert c \Vert_1
\]
where the infimum is taken over all $m$-cycles $c \in C(X;\AA)$ with coefficients in~$\AA$ representing~${\bf a}$.
\end{definition}

We present a couple of classical results, including the proofs for the sake of completeness.

\begin{lemma} \label{lem:RQ}
Every homology class ${\bf a} \in H_m(X;\Z)$ satisfies
\[
\Vert {\bf a} \Vert_{\Delta} = \Vert {\bf a} \Vert_{\Delta}^{\Q}.
\]
\end{lemma}
\begin{proof}
Let $\sigma$ (resp. $\sigma'$) be a real (resp. rational) $m$-cycle representing the real homology class induced by~${\bf a}$.
The difference $\sigma-\sigma'$ is the boundary of an $(m+1)$-chain $c \in C_{m+1}(X,\R)$, that is,
\[
\sigma-\sigma' = \partial c.                
\]
By density of~$\Q$ in~$\R$, there is a rational $(m+1)$-chain~$c'$ (with the same support)
such that $\Vert c - c' \Vert_1$ is arbitrarily small.
Since the boundary of every $(m+1)$-simplex is formed of $m+2$ simplices of dimension~$m$, we have 
\[
\Vert \partial z \Vert_1 \leq (m+2) \, \Vert z \Vert_1
\]
for every $(m+1)$-chain $z  \in C_{m+1}(X; \R)$.
Thus, 
\[
\Vert \sigma-(\sigma' + \partial c') \Vert_1 \leq \Vert \partial(c - c') \Vert_1 \leq (m+2) \, \Vert c - c' \Vert_1
\]
is arbitrarily small.
Therefore, the real cycle~$\sigma$ and the rational cycle $\sigma' + \partial c'$, which both represent the real homology class induced by~${\bf a}$, have arbitrarily closed $\Vert \cdot \Vert_1$-semi-norms.
Hence the result.
\end{proof}

\begin{proposition} \label{prop:SK}
Every homology class ${\bf a} \in H_m(X;\Z)$ satisfies
\begin{align}
\Vert {\bf a} \Vert_\Delta^{\Z} & =  \kappa({\bf a}) \label{eq:kappa1} \\
\Vert {\bf a} \Vert_\Delta & =  \kappa^\infty({\bf a}) \label{eq:kappa2}.
\end{align}
\end{proposition}
\begin{proof}
The inequality $\Vert {\bf a} \Vert_\Delta^{\Z}  \leq \kappa({\bf a})$ is obvious and the reverse inequality $\kappa({\bf a}) \leq \Vert {\bf a} \Vert_\Delta^{\Z} $ can be found in \cite[Proposition~2.1]{LP12} (and also follows from~\cite[p.~108-109]{hatcher}).
Hence the relation~\eqref{eq:kappa1}.

Applying the average procedure of~\eqref{eq:complexite:moyenne1} to the obvious inequality $\Vert {\bf a} \Vert_\Delta \leq \kappa({\bf a})$ yields the bound $\Vert {\bf a} \Vert_\Delta \leq \kappa^\infty({\bf a})$.
For every $\varepsilon > 0$, we also have 
$$
\frac{\kappa(n \, {\bf a})}{n} = \frac{\Vert n \, {\bf a} \Vert_\Delta^{\Z}}{n} \leq \Vert {\bf a} \Vert_\Delta^{\Q} + \varepsilon
$$ 
for some positive integer~$n$, where the first equality follows from~\eqref{eq:kappa1}.
Thus, $\kappa^\infty({\bf a}) \leq \Vert {\bf a} \Vert_\Delta^{\Q} +\varepsilon$ by subadditivity of the function~$\kappa(n \, {\bf a})$ with respect to~$n$.
By Lemma~\ref{lem:RQ}, this yields the bound $\kappa^\infty({\bf a}) \leq \Vert {\bf a} \Vert_{\Delta} + \varepsilon$
for every~$\varepsilon > 0$.
Hence the relation~\eqref{eq:kappa2}.
\end{proof}

\subsection{Universal realization of homology classes.}

\mbox{ }

\medskip


Let us introduce some results in geometric topology following \cite{gaifullin08a}, \cite{gaifullin08b} and~\cite{gaif13}.

\begin{definition} \label{def:permu}
The \emph{$m$-permutahedron}~$\Pi^m$ is the convex hull of the $(m+1)!$ points obtained by permutations of the coordinates of the point $(1,2,\dots,m+1)$ of~$\R^{m+1}$.
It is an $m$-dimensional simple convex polytope of~$\R^{m+1}$ with $2^{m+1}-2$ facets, \ie, $(m-1)$-faces, that lies in the hyperplane
\[
x_1 + \cdots + x_{m+1} = \sum_{j=1}^{m+1} j = \frac{(m+1)(m+2)}{2}.
\]
Here, an $m$-polytope is simple if each of its vertices is contained in exactly $m$ facets.

From a more geometric point of view, see~\cite{gaifullin13}, the $m$-permutahedron~$\Pi^m$ can be obtained by truncating the standard simplex~$\Delta^m \subseteq \R^{m+1}$ given by
\[
x_1 + \cdots + x_{m+1} = 1
\]
with $x_i \geq 0$ as follows.
First, truncate the vertices of~$\Delta^m$ by the hyperplanes $x_i=1-\frac{1}{4}$.
Then, truncate the edges of~$\Delta^m$ by the hyperplanes 
\[
x_{i_1} + x_{i_2} = 1- \left( \frac{1}{4} \right)^2.
\]
At the $k$-th step, truncate the $(k-1)$-faces of~$\Delta^m$ by the hyperplanes
\[
x_{i_1} + \cdots + x_{i_k} = 1 - \left( \frac{1}{4} \right)^k.
\]
The resulting polytope is combinatorially equivalent to the $m$-permutahedron~$\Pi^m$.
The faces of~$\Pi^m$ correspond to the faces of~$\Delta^m$ after truncation of which they appear.

Consider the natural piecewise linear map $\Theta:\Pi^m \to \Delta^m$ which takes every face~$F_\omega$ of~$\Pi^m$ to its corresponding face~$\Delta_\omega$ in~$\Delta^m$.
Note that $\Theta$ is a degree one map which is injective in the interior of~$\Pi^m$.
\end{definition}

\begin{definition} \label{def:tomei}
Consider the manifold~$M_0$ of real symmetric tridiagonal matrices in~$\mathcal{M}_{m+1}(\R)$ with eigenvalues $\lambda_i=i$, for $i=1,\dots,m+1$.
Here, a matrix $A=(a_{i,j})$ is tridiagonal if $a_{i,j}=0$ whenever $|i-j|>1$.
The manifold~$M_0$ will be referred to as the \emph{isospectral $m$-manifold}.
\end{definition}

It was proved by C.~Tomei~\cite{tomei} that the isospectral manifold~$M_0$ is an orientable closed aspherical $m$-manifold.
By \cite{gaifullin08a}, \cite{gaifullin08b} and~\cite{gaif13}, the isospectral $m$-manifold~$M_0$ is tiled by $2^m$ copies of the $m$-permutahedron~$\Pi^m$.
More precisely, the manifold~$M_0$ can be decomposed as 
\[
M_0 = (\Z_2^m \times \Pi^m) / \! \! \sim
\] 
where the equivalence relation is generated by $(s,x) \sim (r_{|\omega|} s ,x)$ whenever $x \in F_\omega$.
Here, the elements~$r_i$ are the standard generators of~$\Z_2^m$.

\medskip

We will rely on the following universal property established by A.~Gaifullin~\cite{gaifullin08a}, \cite{gaifullin08b} regarding Steenrod's problem and the realization of cycles by closed manifolds.

\begin{theorem}[\cite{gaifullin08a}, \cite{gaifullin08b}] \label{theo:gaifullin}
Let $X$ be a path-connected topological space.
Then for every homology class ${\bf a} \in H_m(X;\Z)$, there exist a connected finite-fold covering $\widehat{M}_0 \to M_0$  of the isospectral manifold and a map $f:\widehat{M}_0 \to X$ such that
\[
f_*([\widehat{M}_0]) = q \, {\bf a}
\]
for some positive integer~$q$.
\end{theorem}

\subsection{Homology norm comparison: upper bound on the volume entropy semi-norm}
\mbox{ }

\medskip

Let us state the main theorem of this section.

\begin{theorem} \label{theo:E<S}
Let $X$ be a path-connected topological space.
For every positive integer~$m$, there exists a constant $C_m >0$ which depends only on~$m$ such that every homology class ${\bf a} \in H_m(X;\Z)$ satisfies
\[
\Vert {\bf a}\Vert_E \leq C_m \, \Vert {\bf a}\Vert_\Delta.
\]
\end{theorem}


\begin{proof}
Let ${\bf a} \in H_m(X;\Z)$.
By Proposition~\ref{prop:SK}, see~\eqref{eq:kappa2}, for every $\e >0$ and every integer~$s$ large enough, there exists a map $h=h_s:P \to X$ from an $m$-dimensional geometric $\Delta$-cycles~$P=P_s$ such that 
\begin{equation} \label{eq:hP}
h_*([P]) = s \, {\bf a}
\end{equation}
\begin{equation} \label{eq:quasi2}
s \Vert  {\bf a} \Vert_\Delta \leq \kappa(P) \leq s(\Vert {\bf a} \Vert_\Delta + \varepsilon).
\end{equation}

The second barycentric subdivision of~$P$ gives rise to a simplicial structure on~$P$, see~\cite{hatcher}.
In general, the complex~$P$ is not connected.
After the second barycentric subdivision, we can take the connected sum of the connected components by omitting out some $m$-simplices and gluing together the components to obtain an orientable connected closed pseudomanifold, still denoted by~$P$.
Note that this operation does not increase the number of $m$-simplices.
Taking a third barycentric subdivision ensures that the simplicial structure admits a regular colouring in $m+1$ colours (that is, any two vertices connected by an edge are of distinct colours).
The pseudomanifold~$P$ with this simplicial structure is denoted by~$Z$.
Since the barycentric subdivision of an $m$-simplex gives rise to $m!$ simplices of dimension~$m$, we obtain
\begin{equation} \label{eq:ZP}
\kappa(Z) \leq (m!)^3 \, \kappa(P).
\end{equation}

By Theorem~\ref{theo:gaifullin}, there exists a map $f:\widehat{M}_0 \to Z$ from a finite covering~$\widehat{M}_0$ of~$M_0$ such that 
\begin{equation} \label{eq:qZ}
f_*([\widehat{M}_0]) = q \, [Z] \in H_m(Z;\Z)
\end{equation}
for some positive integer~$q$.

\medskip

Consider the piecewise flat metric on~$M_0$ where all permutahedra are isometric to the standard permutahedron~$\Pi^m$ with its natural Euclidean metric.
The volume of~$M_0$ is equal to $2^m v_m$, where $v_m$ is the Euclidean volume of~$\Pi^m$.

By construction, see \cite{gaifullin08a}, \cite{gaifullin08b} and~\cite{gaif13}, the map $f:\widehat{M}_0 \to Z$ satisfies the following features.
The map $f:\widehat{M}_0 \to Z$ takes every permutahedron~$\Pi^m$ of~$\widehat{M}_0$ to a simplex~$\Delta^m$ of~$Z$ and its restriction to~$\Pi^m$ agrees with the natural piecewise linear map $\Theta:\Pi^m \to \Delta^m$ introduced in Definition~\ref{def:permu}.
Furthermore, the number of permutahedra of the covering~$\widehat{M}_0$ is equal to $q \, \kappa(Z)$, 
where $q$ is the degree of~$f$, see the end of the proof of Proposition~5.3 of~\cite{gaif13}.
Therefore, the volume of~$\widehat{M}_0$ satisfies
\begin{equation} \label{eq:volK}
\vol(\widehat{M}_0) = q \, \kappa(Z) \, v_m
\end{equation}
where $v_m$ is the Euclidean volume of~$\Pi^m$.

\medskip

Consider the composite map $\varphi=h \circ f: \widehat{M}_0 \to Z \simeq P \to X$.
We derive from~\eqref{eq:qZ} and~\eqref{eq:hP} that 
\[
\varphi_*([\widehat{M}_0]) = qs \, {\bf a}.
\]
By definition of the volume entropy semi-norm, we have
\[
\Vert qs \, {\bf a} \Vert_E \leq \ent_\varphi(\widehat{M}_0)^m \, \vol(\widehat{M}_0).
\]
It follows from~\eqref{eq:volK}, \eqref{eq:ZP} and~\eqref{eq:quasi2} that
\[
\vol(\widehat{M}_0) \leq q \, (m!)^3 \, s(\Vert {\bf a} \Vert_\Delta + \e) \, v_m.
\]
Since $\ent_\varphi(\widehat{M}_0) \leq \ent(M_0)$, we deduce that
\[
qs \, \Vert {\bf a} \Vert_E \leq q s \, (m!)^3 \, C'_m \, (\Vert {\bf a} \Vert_\Delta + \e)
\]
where $C'_m = \ent(M_0)^m \, v_m$ is a constant which only depends on~$m$.
Simplifying by $qs$ and letting~$\e$ go to zero, we obtain
\[
\Vert {\bf a} \Vert_E \leq C_m \, \Vert {\bf a} \Vert_\Delta
\]
where $C_m = (m!)^3 \, C_m'$.
\end{proof}

\begin{remark}
An estimate on the volume entropy $\ent(M_0)$ of the isospectral $m$-manifold provides an estimate on the constant~$C_m$ in Theorem~\ref{theo:E<S}.
\end{remark}

\forget
The following result provides an estimate on the constant~$C_m$ in Theorem~\ref{theo:E<S}.

\begin{proposition} \label{prop:entM0}
Let $M_0$ be the isospectral $m$-manifold endowed with the piecewise flat metric where all the permutahedra in the permutahedron decomposition of~$M_0$ are endowed with their standard Euclidean structure.
Then
\[
\ent(M_0) \leq ???
\]
\end{proposition}

\begin{proof}
Consider the graph $\Gamma_0 \subset M_0$ dual to the decomposition of~$M_0$ into $2^m$ permutahedra with the induced length metric.
The graph~$\Gamma_0$ is regular of valence the number $\delta_m=2^{m+1} -2$ of facets of~$\Pi^m$ and has $2^m$ vertices.
Furthermore, the length~$l_m$ of its edges is equal to twice the maximal radius of a ball contained in~$\Pi^m$ centered at the barycenter of~$\Pi^m$.
Thus,
\[
\length(\Gamma_0) = 2^{m-1} \, \delta_m \, l_m.
\]
Since $\Gamma_0$ is a $\delta_m$-regular graph with edge length~$l_m$, the entropy of~$\Gamma_0$ is equal to
\[
\ent(\Gamma_0) = \frac{1}{l_m} \, \log(\delta_m -1).
\]
Consider the lift~$\overline{\Gamma}_0 \subseteq \widetilde{M}_0$ of~$\Gamma_0$ to the universal covering~$\widetilde{M}_0$ of~$M_0$ with the induced length metric.
Note that $G=\pi_1(M_0)$ acts both on~$\widetilde{M}_0$ and~$\overline{\Gamma}_0$.
Since $M_0$ is compact, its universal covering~$\widetilde{M}_0$ is quasi-isometric to~$\overline{\Gamma}_0$, see REF???.
More precisely, there exist some constants $A_0>1$ and $B_0 >0$ such that for every $\gamma \in G$ and every basepoint $x_0 \in \overline{\Gamma}_0$, we have
\begin{equation} \label{eq:quasi}
A_0^{-1} \, d_{\overline{\Gamma}_0}(x_0,\gamma \cdot x_0) - B_0 \leq d_{\widetilde{M}_0}(x_0,\gamma \cdot x_0) \leq d_{\overline{\Gamma}_0}(x_0,\gamma \cdot x_0).
\end{equation}
Note that $A_0$ and~$B_0$ only depend on~$m$. BE EXPLICIT???

Thus, the exponential growth rate of the $G$-orbit of~$x_0$ in~$\widetilde{M}_0$ is bounded by $A_0$ times the exponential growth rate of the $G$-orbit of~$x_0$ in~$\overline{\Gamma}_0$ (and so by~$A_0 \, \ent(\Gamma_0)$).
Therefore,
\[
\ent(\widetilde{M}_0) \leq A_0 \, \ent(\Gamma_0).
\]
\end{proof}
\forgotten

\subsection{Homology norm comparison: lower bound on the volume entropy semi-norm}

\mbox{ } 

\medskip

Let us show a reverse inequality to Theorem~\ref{theo:E<S}.

\begin{theorem} \label{theo:E>S}
Let $X$ be a path-connected topological space.
For every positive integer~$m$, there exists a constant $c_m >0$ which depends only on~$m$ such that every homology class ${\bf a} \in H_m(X;\Z)$ satisfies
\[
\Vert {\bf a}\Vert_E \geq c_m \, \Vert {\bf a}\Vert_\Delta.
\]
\end{theorem}

In order to prove this theorem, we will need the following classical interpretation of the simplicial volume in terms of bounded cohomology, see~\cite{gro82} for the definitions.

\begin{proposition} \label{prop:dual}
Let $X$ be a topological space.
Then every homology class ${\bf a} \in H_m(X;\R)$ satisfies
\[
\Vert {\bf a} \Vert_\Delta = \sup \left\{ \frac{1}{\Vert \alpha \Vert_\infty} \mid \alpha \in H_b^m(X;\R), \langle \alpha , {\bf a} \rangle = 1 \right\}
\]
where $H_b^m(X;\R)$ represents the bounded cohomology of~$X$.
\end{proposition}

The following result is a technical extension of M.~Gromov's inequality~\eqref{eq:gro}.

\begin{proposition} \label{prop:Gineq}
Let $X$ be a path-connected topological space.
Let $\Phi:M \to X$ be a map defined on an orientable connected closed $m$-manifold~$M$.
Then there exists a constant $c_m >0$ depending only on~$m$ such that 
\[
\Omega_{\ker \Phi_*}(M) \geq c_m \, \Vert \Phi_*([M]) \Vert_\Delta.
\]
\end{proposition}

\begin{proof}
Fix a Riemannian metric~$g$ on~$M$.
Let $\bar{M} \to M$ be the covering of~$M$ with fundamental group~$\ker \Phi_*$.
The quotient group $\Gamma=\pi_1(\bar{M})/\ker \Phi_*$ acts by deck transformations on~$\bar{M}$.
Denote by $\mathcal{M}(\bar{M})$ the Banach space of finite measures~$\mu$ on~$\bar{M}$ with the norm
\[
\Vert \mu \Vert = \int_{\bar{M}} | \mu |.
\]
Denote also by $\mathcal{M}^+(\bar{M}) \subseteq \mathcal{M}(\bar{M})$ the cone of positive measures.
Following~\cite[\S2.4]{gro82}, a \emph{smoothing operator} on~$\bar{M}$ is a smooth $\Gamma$-equivariant map $\mathscr{S}:\bar{M} \to \mathcal{M}^+(\bar{M})$.
Define
\[
[\mathscr{S}] = \sup_{x \in \bar{M}} \frac{\Vert d_x \mathscr{S} \Vert}{\Vert \mathscr{S}(x) \Vert}.
\]

Let $\alpha \in H_b^m(X;\R)$ such that $\langle \alpha , \Phi_*([M]) \rangle = 1$, where $\langle \cdot , \cdot \rangle$ is the the bilinear pairing between cohomology and homology given by the Kronecker product.
Define $\beta = \Phi^*(\alpha) \in H^m_b(M;\R)$.
Clearly, $\Vert \beta \Vert_\infty \leq \Vert \alpha \Vert_\infty$ and $\langle \beta , [M] \rangle = 1$.
By~\cite[Proposition, p.33]{gro82}, there exists a closed $m$-form~$\omega$ on~$M$ representing the cohomology class $\beta \in H^m(M;\R)$ such that
\begin{equation} \label{eq:comassS}
\Vert \omega \Vert \leq m! \, \Vert \beta \Vert_\infty \, [\mathscr{S}]^m
\end{equation}
for every smoothing operator $\mathscr{S}: \bar{M} \to \mathcal{M}^+(\bar{M})$.

For $\lambda > \ent_{\ker \Phi*}(M)$, define $\mathscr{S}=\mathscr{S}_{\lambda,R}: \bar{M} \to \mathcal{M}^+(\bar{M})$ as
\[
\mathscr{S}(x) = \left( e^{-R d_{\bar{g}}(x,y)} - e^{-\lambda R} \right) \, \mathbbm{1}_{B_{\bar{g}}(x,R)} \, {\rm dvol}_{\bar{g}}(x).
\]
By~\cite[\S2.5]{gro82}, see also~\cite{BK} for further details, there exists a positive constant~$A_m$ depending only on~$m$ such that 
\begin{equation} \label{eq:Slambda}
[\mathscr{S}] \leq A_m \, \lambda
\end{equation}
for $R$ large enough.
Technically speaking, the bound is stated in~\cite{gro82} and~\cite{BK} when $\bar{M}$ is the universal covering of~$M$, but the proof is exactly the same for intermediate coverings.

Integrating~$\omega$ on~$M$ using the relation~$\langle \omega , [M] \rangle = 1$ and the combination of~\eqref{eq:comassS} with the bounds $\Vert \beta \Vert_\infty \leq \Vert \alpha \Vert_\infty$ and~\eqref{eq:Slambda}, we obtain
\[
1 \leq m! \, (A_m)^m \, \Vert \alpha \Vert_\infty \, \ent_{\ker \Phi_*}(M)^m \, \vol(M).
\]
Hence,
\[
c_m \, \Vert \Phi_*([M]) \Vert_\Delta \leq \Omega_{\ker \Phi_*}(M)
\]
by Proposition~\ref{prop:dual}, with $c_m = (m! \, (A_m)^m)^{-1}$, where $A_m$ is the multiplicative constant in~\eqref{eq:Slambda}.
\end{proof}

We can now proceed to the proof of Theorem~\ref{theo:E>S}.

\begin{proof}[Proof of Theorem~\ref{theo:E>S}]
For every $\varepsilon >0$, there exists a positive integer~$k$ such that
\[
\Omega(k \, {\bf a}) \leq k \, (\Vert {\bf a} \Vert_E + \varepsilon).
\]
Thus, there exists a map $\varphi:P \to X$ defined on an orientable connected closed $m$-pseudomanifold~$P$ such that $\varphi_*([P]) = k \, {\bf a}$ and
\begin{equation} \label{eq:OE}
\Omega_{\ker \varphi_*}(P) \leq \Omega(k \, {\bf a}) + \varepsilon  \leq k \, \Vert {\bf a} \Vert_E + (k+1) \varepsilon.
\end{equation}

By Thom's theorem, there exists a map $f:M \to P$ defined on an orientable connected closed $m$-manifold~$M$ such that
\[
f_*([ M]) = d \, [P] \in H_m(P; \Z)
\]
for some suitable integer~$d$.
Extend $f:M \to P$ by handle attachements into a $\pi_1$-surjective map $f':M' \to P$ where
\[
M' = M \sharp \left( \sharp_{i=1}^l S^1 \times S^{m-1} \right).
\]
Clearly, $f'_*([M']) = f_*([M]) = d \, [P]$.
By~\cite[Proposition~2.2]{B93}, we have
\begin{equation} \label{eq:OO}
\Omega_{\ker \Phi_*}(M') \leq d \, \Omega_{\ker \varphi_*}(P)
\end{equation}
where $\Phi:M' \to X$ is the composite map $\Phi=\varphi \circ f'$.

Now observe that $\Phi_*([M']) = dk \, {\bf a}$.
By Proposition~\ref{prop:Gineq}, we derive 
\begin{equation} \label{eq:SO}
c_m \, dk \, \Vert {\bf a} \Vert_\Delta \leq \Omega_{\ker \Phi_*}(M').
\end{equation}

Combining the inequalities~\eqref{eq:OE}, \eqref{eq:OO} and~\eqref{eq:SO}, dividing by~$dk$ and letting~$\varepsilon$ go to zero, we obtain
\[
c_m \, \Vert {\bf a} \Vert_\Delta \leq \Vert {\bf a} \Vert_E
\]
as desired.
\end{proof}

\subsection{Density of the volume entropy semi-norm of manifolds}

\mbox{ } 

\medskip

The following density result can be deduced from the equivalence of the semi-norms given by Theorem~\ref{theo:main} with the recent work~\cite{HL19}.

\begin{corollary}
Let $m \geq 4$ be an integer.
Then the set of all volume entropy semi-norms~$\Vert M \Vert_E$ of orientable connected closed $m$-manifolds~$M$ is dense in~$[0,\infty)$.
\end{corollary}

\begin{proof}
Fix $\varepsilon >0$.
By~\cite[Theorem~A]{HL19}, there exists an orientable connected closed $m$-manifold~$M$ with $0 < \Vert M \Vert_\Delta < \frac{\varepsilon}{C_m}$ where $C_m$ is the positive constant in Theorem~\ref{theo:main}.
Define
\[
M_k = M \sharp \cdots \sharp M
\]
as the connected sum of $k$ copies of~$M$.
By additivity of the simplicial volume under connected sums, see~\cite{gro82}, we have $\Vert M_k \Vert_\Delta = k \, \Vert M \Vert_\Delta$.
It follows from the equivalence of the semi-norm, see Theorem~\ref{theo:main}, that the sequence~$\Vert M_k \Vert_E$ starts from the interval~$(0,\varepsilon)$ and goes to infinity.
Furthermore, $0 < \Vert M \Vert_E < \varepsilon$.
Now, the map $M_{k+1} \to M_k$ collapsing one copy of~$M$ to a point is of degree~$1$.
By the functorial properties of the volume entropy semi-norm, namely \eqref{f3} and~\eqref{f5} of Theorem~\ref{theo:functorial}, we derive
\[
\Vert M_k \Vert_E \leq \Vert M_{k+1} \Vert_E \leq \Vert M_k \Vert_E + \Vert M \Vert_E.
\]
Thus, the sequence~$\Vert M_k \Vert_E$ is nondecreasing and increases by at most $\Vert M \Vert_E < \varepsilon$ at each step.
We deduce from the properties of the sequence~$\Vert M_k \Vert_E$ that every interval of~$[0,\infty)$ of length~$\varepsilon$ contains at least one term~$\Vert M_k \Vert_E$ .
Hence the desired density result.
\end{proof}

\forget

\section{Estim\'e inf\'erieure pour l'entropie volumique}

EN ATTENDANT.

\begin{theorem}
Soit $X$ un espace topologique connexe par arcs dont $\pi_1(X)$ est de pr\'esentation finie. 
Alors pour tout $m \geq 4$
et pour toute classe homologique ${\bf a} \in H_m(X, \R)$ on a
$$
c_m \Vert {\bf a}\Vert_{\Delta} \leq \Vert {\bf a}\Vert_E,
$$
o\`u $c_m >0$ est une constante  ne d\'ependant que de la dimension $m$.
\end{theorem}
\begin{proof}
On consid\`ere ${\bf a} \in H_m(X, \R)$ et soit $\varepsilon > 0$. Il existe un entier naturel $k$ tel que
$$
\Omega(k{\bf a}) \leq k(\Vert {\bf a}\Vert_E + \varepsilon).
$$
Il existe une pseudo-vari\'et\'e orientable $m$-dimensionnelle $P$ et une application $f: P \longrightarrow X$
telle que $f_*([P]) = k{\bf a}$ et de plus
\begin{equation}\label{eq:1}
\Omega_{\ker f_*}(P) \leq \Omega(k{\bf a}) + \varepsilon \leq  k\Vert {\bf a}\Vert_E + (k+1)\varepsilon.
\end{equation}

Par le th\'eor\`eme de Thom il existe une vari\'et\'e orientable $M$ et une application $f': M \longrightarrow P$
telles que $f'_*([ M]) = d[P] \in H_m(P, \Z)$ o\`u $d$ est un nombre naturel appropri\'e. 
Soit $\{a_1, a_2, \dots , a_l\}$ un syst\`eme de g\'en\'erateurs du $\pi_1(P)$. Consid\'erons la nouvelle
vari\'et\'e
$$
M' = M \mathop{\sharp}\limits_{i=1}^l(S^{m-1}\times S^1)_i
$$
et son application $f': M \longrightarrow P$ donn\'ee \`a l'aide de la d\'ecomposition suivante :
$$
\begin{CD}
M' \longrightarrow M \mathop{\bigvee}\limits_{i=1}^l(S^1)_i  
@>{f\mathop{\bigvee}\limits_{i=1}^l a_i}>> P,
\end{CD}
$$
ou $a_i$ sont vues comme des applications $a_i: S^1 \longrightarrow P$ correspondantes au g\'en\'erateur $a_i$.
La paire $(M', f')$ v\'erifie :
$$
f'_*([M']) = d[P] , \ \ f'_* : \pi_1(M') \longrightarrow \pi_1(P) \ \ \text{est un epimorphisme}.
$$
En appliquant le truc r\'elatif  de Hopf \cite{Brunnbauer08} nous pouvons compter que $f'$ est d\'ej\`a une
application $(m, d)$-monotone, la Proposition 2.2 de \cite{B93} implique donc
\begin{equation}\label{eq:2}
\Omega_{\ker (f'\circ f)_*}(M') \leq d\Omega_{\ker f_*}(P) \leq  dk\Vert {\bf a}\Vert_E + d(k+1)\varepsilon.
\end{equation}
Pour finir la d\'emonstration du Th\'eor\`eme nous avons besoin le lemme suivant
\begin{lemma}\label{lemme:1}
Sous les conditions du th\'eor\`eme il existe une nouvelle vari\'et\'e $N$ orientable de dimension $m$ et une application $h: N \longrightarrow X$ tels que:

\noindent 1. $h_*$ induit un isomorphisme des groupes fondamentaux;

\noindent 2. $h_*([N]) = (f\circ f)'_*([M'])$ dans $H_m(X, \Z)$;

\noindent 3. $\Omega(N) \leq   \Omega_{\ker (f'\circ f)_*}(M')$.
\end{lemma}
On conclut la d\'emonstration du th\'eor\`eme. Par le r\'esultat de M.~Gromov \cite{gro82} il existe une constante
positive $c_m$ ne d\'ependant que de la dimension telle que pour toute vari\'et\'e $m$-dimensionnelle $N$
on a 
$$
c_m \Vert N \Vert_{\Delta} \leq \Omega(N). 
$$
Pour la vari\'et\'e $N$ on tire du lemme $h_*([N]) = (f\circ f)'_*([M']) = dk{\bf a}$ d'ou 
$\Vert N \Vert_{\Delta} \geq \Vert dk{\bf a} \Vert_{\Delta} = dk\Vert {\bf a} \Vert_{\Delta}$.

Cette in\'egalit\'e avec (\ref{eq:2}) et 3) du lemme impliquent finalement
$$
c_m dk\Vert {\bf a} \Vert_{\Delta} \leq  dk\Vert {\bf a}\Vert_E + d(k+1)\varepsilon,
$$ 
d'ou $c_m \Vert {\bf a} \Vert_{\Delta} \leq  \Vert {\bf a}\Vert_E + 2\varepsilon$
Comme $\varepsilon$ est arbitraire ceci implique le r\'esultat.
\end{proof}

\begin{proof}[Proof Lemma \ref{lemme:1}]
On consid\`ere une r\'epresentation 
\begin{equation}\label{eq:3}
\mathcal{R} = \langle s_1, s_2, \dots s_u \vert r_1, r_2, \dots r_v \rangle
\end{equation}
du groupe fondamental $\pi_1(X)$ et soit $K = K(\mathcal{R})$ le $CW$-complexe deux dimensionnelle ontenu
par cette repr\'esentation. Posons $q : K \longrightarrow X$ une application induisant isomorphisme des groupes fondamentaux.

Le complexes $T = M' \bigvee K$ poss\`ede des propri\'et\'es suivantes :

\noindent 1. l'application $f_1 \bigvee q: M' \bigvee K \longrightarrow X $, ou $f_1 = f\circ f'$,
est surjective sur groupes fondamentaux;

\noindent 2. $T$ a une $CW$-structure avec une unique $m$-cellule dont la classe fondamentale est bien d\'efinie.
On suppose que cette classe co\"{\i}ncide avec $[M']$ et alors $(f\bigvee q)_* ([T]) = dk{\bf a}$; 

\noindent 3. le th\'eor\`eme 2.6 implique $\Omega_{\ker(f_1\bigvee q)_*}(T) = \Omega_{\ker {f_1}_*}(M')$. 

On consid\`ere la vari\'et\'e suivante 
\begin{equation}\label{eq:4}
\mathop{\sharp}\limits_{i=1}^u(S^{m-1}\times S^1)_i
\end{equation}
et soit $M_1$ la vari\'et\'e obtenue de (\ref{eq:4}) par des chirurgies sph\'eriques le-long des courbes
fermes simples choisies selon les r\'elations (\ref{eq:3}). Il existe une application \'evidente 
$$
p : M_1 \longrightarrow K
$$
qui contracte toutes anses du type $(1, m-1)$ sur les cercles de $K$ et toutes les anses du type $(2, m-2)$
sur les $2$-disques de $K$. Cette application induit un isomorphisme de groupes fondamentaux.

On consid\`ere ensuite la vari\'et\'e $M_2 = M' \sharp M_1$ et l'application \'evidente de degr\'e 1
$f_2 : M_2 \longrightarrow T$ qui induit un isomorphisme des groupes fondamentaux. Alors
$((f_1 \bigvee q)\circ  f_2)_*([M_2]) =  dk{\bf a}$

Pour obtenir la vari\'et\'e d\'esir\'ee $N$ on fait quelques chirurgies sph\'eriques sur $M_2$. Pour ce but
on choisit un syst\`eme de g\'en\'erateurs $\{c_1, c_2, \dots c_n \}$ du $\pi_1(M')$ et soient
$$
{f_1}_*(c_i) = \omega_i = \omega_i(s_1, s_2, \dots ,s_u), \ \ i = 1, 2, \dots , u , \ \ \text{dans} \ \ \pi_1(X) ,
$$
o\`u $\omega_i$  sont des mots en g\'n\'erateurs (\ref{eq:3}). On consid\`ere dans $M_2$ les courbes 
ferm\'ees simples $\{\gamma_i \}_{i=1}^u$ homotopes aux \'elements $\{c_i \omega_i^{-1} \}_{i=1}^u$
du groupe fondamental $\pi_1(M_2) = \pi_1(M') \ast \pi_1(M_1)$. Posons
\begin{equation}\label{eq:5}
Y = M_2 \mathop{\bigcup}\limits_{\gamma_i}(D^2)_i,
\end{equation}
ou $D^2$ signifie le disque 2-dimensionnelle.
Le choix de courbes de recollement $\{\gamma_i \}_{i=1}^u$ nous assure que l'application 
$$
(f_1 \bigvee q)\circ  f_2 : M_2 \longrightarrow X
$$
se prolonge en une application $f_3: Y \longrightarrow X$ et que cette derni\`ere application induit 
un isomorphisme des groupes fondamentaux. Le principe d'extension nous nous assure que
\begin{equation}\label{eq:6}
\Omega(Y) = \Omega_{\ker((f_1\bigvee q)\circ f_2}(M_2) \leq 
\Omega_{\ker(f_1\bigvee q)_*}(T) = \Omega_{\ker {f_1}_*}(M').
\end{equation}
Pour obtenir la vari\'et\'e $N$ on ne reste que \`a faire des chirurgies sph\'eriques du type $(2, m-2)$ de  $M_2$
le longue les courbes $\{\gamma_i \}_{i=1}^u$. La contraction des anses sur les disques de (\ref{eq:5}) nous 
provient une application monotone $N \longrightarrow Y$ et avec (\ref{eq:6}) on obtient le r\'esultat.
\end{proof}
\begin{remark}
Dans les dimensions 2 et 3 toute classe homologique enti\`ere
est pr\'esentable par une vari\'et\'e. N\'eanmoins 
l'application des groupes fondamentaux correspondante \`a cette pr\'esentation poss\`ede un noyau non trivial.
En g\'en\'eral l'existence de ce noyau est in\'evitable dans les dimensions 2 et 3 
compte tenu le fait que les groupes fondamentaux de vari\'et\'es dans ces petites dimensions sont particuliers.
\end{remark}

\forgotten

\section{Systolic volume of homology classes}



In this section, we bound from above the systolic volume of the multiple of a given homology class.
This positively answers a conjecture of~\cite{BB15}, where a sublinear upper bound was established. 
This result also allows us to define the systolic semi-norm in real homology.


\subsection{Systolic volume of a multiple homology class}

\begin{theorem} \label{theo:ka}
Let $X$ be a path-connected topological space.
Then for every homology class ${\bf a} \in H_m(X;\Z)$, there exists a constant $C=C({\bf a}) >0$ such that for every $k \geq 2$, we have
\[
\sigma(k \, {\bf a}) \leq C \, \frac{k}{(\log k)^m}.
\]
\end{theorem}


The proof of Theorem~\ref{theo:ka} rests on some systolic estimates in geometric group theory based on the following notion.

\begin{definition}
Let $G$ be a finitely generated group and $S$ be a finite generating set of~$G$.
Denote by~$d_S$ the word distance induced by~$S$.
For every finite index subgroup $\Gamma \subseteq G$, define
\[
\sys(\Gamma,d_S) = \inf_{\gamma \in \Gamma \setminus \{ e \}} d_S(e,\gamma).
\]
\end{definition}

The systolic growth of finitely generated linear groups has been described by K.~Bou-Rabee and Y.~Cornulier, see~\cite{corn}.
Originally stated in terms of residual girth rather than in terms of systolic growth, their result can be written as follows.

\begin{theorem} \label{theo:corn}
Let $G$ be a finitely generated linear group over a field and $S$ be a finite symmetric generating set of~$G$.
Then there exists a sequence of subgroups $\Gamma_k \subseteq G$ of finite index~$k$ such that
\[
\sys(\Gamma_k,d_S) \geq C_0 \, \log k
\]
for some $C_0>0$ which does not depend on~$k$.
\end{theorem}
\begin{remark}
A similar estimate has been previously stated without proof by M.~Gromov for finitely generated subgroups~$G$ of~${\rm SL}_d(\Z)$ under the extra assumption that no unipotent element lies in~$G$, see~\cite[Elementary Lemma, p.~334]{gro96}.
\end{remark}

We need to review some features of the isospectral $m$-manifold~$M_0$ introduced in Definition~\ref{def:tomei}. 

\medskip

It was proved by C.~Tomei~\cite{tomei} that $M_0$ is an orientable closed aspherical $m$-manifold.
By M.~Davis~\cite{davis}, its fundamental group~$G = \pi_1(M_0)$ is isomorphic to a torsion-free subgroup of finite index of the Coxeter group
\[
\begin{split}
W = \langle s_1,\dots,s_m,r_1,\dots,r_m \mid s_i^2 = r_i^2 =1, s_i s_j = s_j s_i \text{ for } |i-j|>1, \\
 s_is_{i+1}s_i = s_{i+1} s_i s_{i+1}, r_ir_j=r_jr_i, s_i r_j = r_j s_i \text{ for } i \neq j \rangle.
 \end{split}
\]
Recall that J.~Tits showed that every Coxeter group admits a faithful linear representation into a finite-dimensional vector space, see~\cite[Chap.~V, \S4, Corollary~2]{bourbaki}.
Thus, the group~$G$ is linear.
This is an important feature in view of Theorem~\ref{theo:corn}.

\medskip

We can now proceed to the proof of Theorem~\ref{theo:ka}.

\begin{proof}[Proof of Theorem~\ref{theo:ka}]
Let $G = \pi_1(M_0)$.
Fix a finite symmetric generating set~$S$ of~$G$ once and for all.
The metric on~$M_0$ induced by the Hilbert-Schmidt metric (also called the Frobenius metric) on~$\mathcal{M}_{m+1}(\R)$ lifts to a metric~$d_0$ on the universal covering~$\widetilde{M}_0$ of~$M_0$.
(Here, the choice of the metric does not matter. We simply fix one once and for all.)
Since $M_0$ is compact, its universal covering~$\widetilde{M}_0$ is quasi-isometric to~$(G,d_S)$, see~\cite[IV.B, Theorem~23]{harpe}.
More precisely, there exist some constants $A_0>1$ and $B_0 >0$ such that for every $\gamma \in G$ and every $x \in \widetilde{M}_0$, we have
\begin{equation} \label{eq:quasi}
A_0^{-1} \, d_S(e,\gamma) - B_0 \leq d_0(x,\gamma \cdot x) \leq A_0 \, d_S(e,\gamma) + B_0.
\end{equation}
Note that $A_0$ and~$B_0$ only depend on~$m$.

By Theorem~\ref{theo:gaifullin}, there exist a map $f:\widehat{M}_0 \to X$ from a finite covering~$\widehat{M}_0$ of~$M_0$ and a positive integer~$q$ such that 
\begin{equation}
f_*([\widehat{M}_0]) = q \, [{\bf a}].
\end{equation}
Let $\Gamma \subseteq \widehat{G} := \pi_1(\widehat{M}_0)$ be a finite index subgroup of~$\widehat{G}$.
Denote by $f_\Gamma: \widetilde{M}_0/\Gamma \to X$ the lift of $f:\widehat{M}_0 \to X$ under the canonical projection $\pi_\Gamma:\widetilde{M}_0/\Gamma \to \widehat{M}_0$.
By the first inequality of~\eqref{eq:quasi}, we have
\begin{equation} \label{eq:systoles}
A_0^{-1} \, \sys(\Gamma,d_S) - B_0 \leq \sys(\widetilde{M}_0/\Gamma) \leq \sys_{f_\Gamma}(\widetilde{M}_0/\Gamma).
\end{equation}

Now, apply Theorem~\ref{theo:corn} about the systolic growth of linear groups to the finitely generated linear group~$\widehat{G}$.
Thus, there exists a sequence of subgroups $\Gamma_k \subseteq \widehat{G}$ of finite index~$[\widehat{G}:\Gamma_k] = k \geq 2$ such that
\[
\sys(\Gamma_k,d_S) \geq C_0 \, \log k
\]
for some~$C_0 >0$ which does not depend on~$k$.

Let $\widehat{M}_k = \widetilde{M}_0/\Gamma_k$.
We derive from~\eqref{eq:systoles} that
\begin{equation} \label{eq:sysf}
\sys_{f_k}(\widehat{M}_k) \geq A_0^{-1} C_0 \, \log k - B_0 \geq D_0 \, \log k
\end{equation}
where $f_k:\widehat{M}_k \to X$ is the lift of $f:\widehat{M}_0 \to X$ and $D_0 >0$ does not depends on~$k$.
Since $\Gamma_k$ is of index~$k$ in~$\widehat{G}$ and the map $f:\widehat{M}_0 \to X$ represents~$q \, {\bf a}$, we deduce that 
\[
(f_k)_*([\widehat{M}_k]) = kq \, {\bf a}.
\]

Since $\vol (\widehat{M}_k)  = k \, \vol(\widehat{M}_0)$, this yields the inequalities
\[
\sigma(kq \, {\bf a}) \leq \sigma_{f_k}(\widehat{M}_k) \leq C' \, \frac{k}{(\log k)^m}
\]
for every $k \geq 2$, where $C'=C'({\bf a})$ does not depend on~$k$.
Since $\sigma$ is sub-additive, see~\cite[Proposition~3.6]{BB15}, we derive
\[
\sigma(k \, {\bf a}) \leq \sigma \left( \ceil[\Big]{\frac{k}{q}} q \, {\bf a} \right) \leq C \, \frac{k}{(\log k)^m}
\]
for every $k \geq 2$, where $C=C({\bf a})$ does not depend on~$k$.
\end{proof}


\subsection{Systolic semi-norm} \label{subsec:systolic}

\mbox{ } 

\medskip

Theorem~\ref{theo:ka} allows us to define the systolic semi-norm in real homology of dimension~$m \geq 3$
This definition is based on the following observation, whose proof is left to the reader.

\begin{lemma}\label{lemma:technique}
Let $\mathfrak{M}$ be a $\Z$-module endowed with a translation-invariant pseudo-distance~$\varrho$.
Given a function $h:\N \to \R_+$ with $\lim_{k \to \infty} h(k) = \infty$, suppose that for every ${\bf a} \in \mathfrak{M}$, there is a positive constant $C=C({\bf a})$ such that $\varrho(0,k \, {\bf a}) \leq C \, h(k)$ for every $k \in \N$.
Then 
\[
\widehat{\varrho}({\bf a}, {\bf b}) = \limsup_{k \to \infty} \frac{\varrho(k \, {\bf a}, k \, {\bf b})}{h(k)}
\]
defines a translation-invariant pseudo-distance on~$\mathfrak{M}$.
\end{lemma}


Consider a path-connected topological space~$X$ and the systolic volume~$\sigma$ defined on the homology group~$H_m(X;\Z)$, see~\eqref{eq:sigmaa}.
By~\cite[Corollary~5.3]{BB15}, the systolic volume induces a translation-invariant pseudo-distance~$\varrho$ on~$H_m(X;\Z)$ with $m \geq 3$, defined by $\varrho({\bf a},{\bf b}) = \sigma({\bf a} - {\bf b})$.
Applying Lemma~\ref{lemma:technique} to~$\varrho$ with $h(k) = \frac{k}{(\log k)^m}$, see Theorem~\ref{theo:ka}, we obtain a new translation-invariant pseudo-distance~$\widehat{\varrho}$ on~$H_m(X;\Z)$.
Denote by~$\widehat{\sigma}({\bf a}) = \widehat{\varrho}(0,{\bf a})$ the distance from the origin.
Applying a stabilization process to~$\widehat{\sigma}$ leads to the following definition of the systolic semi-norm.

\begin{definition}
For every ${\bf a} \in H_m(X;\Z)$, define
\[
\Vert {\bf a} \Vert_\sigma = \mathop{\lim}\limits_{k \rightarrow \infty}\frac{\widehat{\sigma}(k \, {\bf a})}{k}.
\]
This functional extends to $ H_m(X, \R) = H_m(X, \Z)\otimes \R$ in a natural way and gives rise to a semi-norm, still denoted by~$\Vert \cdot \Vert_\sigma$, on~$H_m(X;\R)$, called the \emph{systolic semi-norm}.
Note that this definition differs from the one proposed in~\cite[\S5.41]{Gromov99}.
\end{definition}

By~\eqref{eq:sigma-lambda}, every class ${\bf a} \in H_m(X;\Z)$ satisfies
\[
\widehat{\sigma}({\bf a}) \geq \lambda_m \, \Vert {\bf a} \Vert_\Delta.
\]
Therefore, the systolic semi-norm is bounded from below by the simplicial volume (up to a multiplicative constant).
More precisely, 
\[
\Vert {\bf a} \Vert_\sigma \geq \lambda_m \, \Vert {\bf a}\Vert_{\Delta}
\]
for every ${\bf a} \in H_m(X;\R)$, where $\lambda_m$ is the multiplicative constant in~\eqref{eq:sigma-lambda}.

Though we did not pursue in this direction, we can raise the question whether the systolic semi-norm is equivalent to the simplicial volume (or the volume entropy semi-norm) for every $m \geq 3$.
Note that there exist closed manifolds with bounded simplicial volume and arbitrarily large systolic volume, see~\cite{sab07}.

\end{document}